\documentclass[12pt]{amsart}
\usepackage{mathtools}
\usepackage{a4wide}
\usepackage{color}
\usepackage{amssymb}
\usepackage{amscd}
\usepackage[all]{xy}
\usepackage{enumitem}
\usepackage{url}
\usepackage{bm}
\usepackage{xcolor}
\usepackage{cleveref}

\newif\ifdebug
\debugfalse
\newcommand{\printname}[1]
   {\smash{\makebox[0pt]{\hspace{-1.0in}\raisebox{8pt}{\tiny #1}}}}
\newcommand{\labell}[1] {\ifdebug {\label{#1}\printname{#1}} 
\else {\label{#1}} \fi}

\def \on {\operatorname}

\def \Id {{\operatorname{Id}}}
\def \std {{\operatorname{std}}}

\def \cut {\operatorname{cut}}
\def \Pcut {P_{\cut}}

\numberwithin{equation}{section}

\newtheorem{theorem}[equation]{Theorem}
\newtheorem{lemma}[equation]{Lemma}
\newtheorem{proposition}[equation]{Proposition}
\newtheorem{corollary}[equation]{Corollary}
\newtheorem*{corollary*}{Corollary}

\theoremstyle{definition}
\newtheorem{definition}[equation]{Definition}
\newtheorem{example}[equation]{Example}

\newtheorem*{claim*}{Claim}

\theoremstyle{remark}
\newtheorem{remark}[equation]{Remark}
\newtheorem{construction}[equation]{Construction}

\def\eor{\unskip\ \hglue0mm\hfill$\diamond$\smallskip\goodbreak}
\def\eoe{\unskip\ \hglue0mm\hfill$\between$\smallskip\goodbreak}
\def\eod{\unskip\ \hglue0mm\hfill$\diamond$\smallskip\goodbreak}
\def\eoc{\unskip\ \hglue0mm\hfill$\diamond$\smallskip\goodbreak}

\setlist{topsep=0pt,itemsep=6pt,parsep=3pt}

\usepackage{epsfig}
\usepackage[all]{xy}

\newcommand{\C}{{\mathbb{C}}}

\newcommand{\Z}{{\mathbb{Z}}}

\newcommand{\R}{{\mathbb{R}}}
\newcommand{\N}{{\mathbb{N}}}
\newcommand{\T}{{\mathbb{T}}}

\def \frakP {{\mathfrak P}}
\def \frakM {{\mathfrak M}}
\def \frakQ {{\mathfrak Q}}
\def \frakA {{\mathfrak A}}
\def \frakB {{\mathfrak B}}

\def \Cx {\C^{\times}}

\def \calO {{\mathcal O}}

\newcommand{\algt}{\mathfrak{t}}

\DeclareMathOperator{\primitive}{primitive}
\def \pre {^{-1}}

\DeclareMathOperator{\free}{free}
\def \Mfree {M_{\free}}

\def \invt {{\operatorname{iso}}}
\def \iso{{\operatorname{iso}}}

\def \loc-iso{{\operatorname{loc-iso}}}

\def \tphi {{\widetilde{\phi}}}
\def \tpsi {{\widetilde{\psi}}}
\def \calF {{\mathcal F}}

\def \intQ {\mathring{Q}}

\def \ol {\overline}
\def \thetabar {\ol{\theta}}

\def \tZ {\algt_{\Z}}
\def \tZhat {\hat{\algt}_{\Z}}
\def \etahat {{\hat{\eta}}}
\def \lambdahat {\hat{\lambda}}

\def \Rplus {\R_{\geq 0}}
\def \Rpos {\R_{>0}}

\def \del {\partial}
\def \del {\partial}
\def \ssminus {\setminus}

\def \actsonup {\  \rotatebox[origin=c]{0}{$\circlearrowright$}\  }

\begin{document}

\title{Classification of locally standard torus actions 
%\\ draft -- \today
}
\date{\today}

\author[Y.\ Karshon]{Yael Karshon}
\email{yaelkarshon@tauex.tau.ac.il}
\address{Tel Aviv University, and the University of Toronto}
\author[S.\ Kuroki]{Shintar\^o Kuroki}
\email{kuroki@ous.ac.jp}
\address{Okayama University of Science}

%\address{OCAMI/University of Toronto}
%\email{kuroki@scisv.sci.osaka-cu.ac.jp}

%\thanks{The author was supported in part by the JSPS 
%Strategic Young Researcher Overseas Visits Program 
%for Accelerating Brain Circulation
%"Deepening and Evolution of Mathematics and Physics, 
%Building of International Network Hub based on OCAMI".}


%\keywords{Torus manifold; Cohomological rigidity}

\begin{abstract}\ 
An action of a torus $T$ on a manifold $M$ is \emph{locally standard} 
if, at each point, the stabilizer is a sub-torus 
and the non-zero isotropy weights are a basis to its weight lattice.
The quotient $M/T$ is then a manifold-with-corners,
decorated by a so-called unimodular labelling, 
which keeps track of the isotropy representations in $M$,
and by a degree two cohomology class with coefficients in the integral lattice
$\algt_\Z$ of the Lie algebra of~$T$, 
which encodes the ``twistedness'' of $M$ over $M/T$.
We classify locally standard smooth actions of~$T$,
up to equivariant diffeomorphism,
in terms of triples $(Q,\lambdahat,c)$, 
where $Q$ is a manifold-with-corners,
$\lambdahat$ is a unimodular labelling, 
and $c$ is a degree two cohomology class with coefficients in $\algt_\Z$.
\end{abstract}

\maketitle
\setcounter{tocdepth}{1}
\tableofcontents

%%%%%%%%%%%%%%%%%%%%%%%%%%%%%%%%%%%%%%%%%%%%%%%%%%%%%%%%%%%%%%%%%%%%%%%%%  
\section{Introduction}
%%%%%%%%%%%%%%%%%%%%%%%%%%%%%%%%%%%%%%%%%%%%%%%%%%%%%%%%%%%%%%%%%%%%%%%%%

\labell{sec:intro}

In this paper, given a torus $T$, 
we classify locally standard $T$-manifolds up to equivariant diffeomorphism.
Along the way, we obtain an equivalence of categories
between locally standard $T$-manifolds and certain principle $T$-bundles.
In this section we give a sketch.
For detailed definitions and statements, see Section \ref{sec:defs}.

A \emph{locally standard $T$-manifold} 
is a (smooth) manifold $M$, equipped with a (smooth) $T$-action,
such that, at each point, the stabilizer is a sub-torus
and the non-zero isotropy weights are a basis to its weight lattice.
Equivalently, $M$ is locally equivariantly diffeomorphic
to $\C^n \times \T^l \times \R^m$ 
where $T$ acts through an isomorphism with $\T^n \times \T^l$.
Special cases include principal $T$-bundles and locally toric $T$-manifolds.
See \Cref{def:locally standard}, \Cref{def:Tchart}
and \Cref{locally standard criterion}, 
and \Cref{locally standard examples}.

The invariants of a locally standard $T$-manifold $M$ are obtained from the 
following three facts.
See \Cref{l:quotient}, \Cref{def:unimodular} and \Cref{l:labelling}, 
and \Cref{l:cM}.

\begin{itemize}[itemsep=6pt]
\item
The quotient $M/T$ is a manifold-with-corners $Q$.
\item
Labelling each point of $Q$ of depth one by the stabilizer of its pre-image
defines a \emph{unimodular labelling} on $Q$.
\item
The quotient map $M \to Q$ restricts
to a principal $T$-bundle over the interior of $Q$;
the \emph{Chern class} of the $T$-manifold $M$
classifies this principal bundle.
\end{itemize}

The classification is then achieved in the following theorem.

\begin{theorem} \labell{thm:main0}
\begin{itemize}[itemsep=6pt]
\item[\scalebox{0.4}{$\blacksquare$}]
A $T$-equivariant diffeomorphism of locally standard $T$-manifolds
descends to a diffeomorphism of their quotients that intertwines
their unimodular labellings and their Chern classes.
\item[\scalebox{0.4}{$\blacksquare$}]
Every such diffeomorphism between such quotients
comes from a $T$-equivariant diffeomorphism
of the locally standard $T$-manifolds.
\item[\scalebox{0.4}{$\blacksquare$}]
Up to isomorphism,
every manifold-with-corners $Q$ that is equipped with a unimodular labelling
and with a class in $H^2(Q;\tZ)$ 
is obtained in this way from a locally standard $T$-manifold.
\end{itemize}
\end{theorem}

\begin{proof}[Pointer to proof]
The three items are detailed in Parts \eqref{main:descends},
\eqref{main:lifts}, and \eqref{main:exists} of \Cref{thm:main},
respectively.
The statement of \Cref{thm:main}
is followed by a proof of Part~\eqref{main:descends},
and by pointers to proofs  of Parts~\eqref{main:lifts} and \eqref{main:exists}, 
which occupy the rest of this paper.
\end{proof}

The first item of \Cref{thm:main0}
exhibits the unimodular labelling and the Chern class
as \emph{invariants} of a locally standard $T$-manifold $M$.
The second item gives the \emph{uniqueness} part of the classification:
these invariants are complete,
namely, they determine $M$ up to equivariant diffeomorphism.
The third item gives the \emph{existence} part of the classification:
any value of the invariants can arise from a locally standard $T$-manifold.

According to \Cref{thm:main0}, the map $M \mapsto M/T$
gives a full and essentially surjective
functor $\frakM_\invt \to \frakQ_\invt$ 
from the category $\frakM_\invt$ of locally standard $T$-manifolds
and their equivariant diffeomorphisms
to the category $\frakQ_\invt$ of manifolds-with-corners $Q$,
decorated by a unimodular labelling $\lambdahat$ 
and with a cohomology class in $H^2(Q;\tZ)$, 
and their decoration-preserving diffeomorphisms.
 %that respect $\lambda$ and $c$.
See \Cref{def:functor} and~\Cref{rk:functor}.

Our classification passes through an equivalence of categories,
from the category $\frakP_\invt$ of principal $T$-bundles 
over decorated manifolds-with-corners
with their decoration-preserving equivariant diffeomorphisms,
%equipped with unimodular labellings, 
to the category $\frakM_\invt$ of locally standard $T$-manifolds
with their equivariant diffeomorphisms.
See \Cref{def:bundle quotient functor} and \Cref{equiv of cat}.
We expect to be able to extend these results to morphisms
that are not necessarily diffeomorphisms; see \Cref{rk:transverse maps}.

We borrow ideas from Karshon-Lerman's classification \cite{KL} 
of not-necessarily-compact symplectic toric manifolds,
which, in turn, borrows ideas from Lerman-Tolman's classification \cite{LT} 
of (effective) symplectic toric orbifolds;
both of these are inspired by Haefliger-Salem's paper \cite{HS}.

In the literature, ``locally standard $T$-action'' often refers 
to a $T$-action that is locally equivariantly diffeomorphic
to $\C^n$, where $T$ acts through an isomorphism
with the standard torus $\T^n$, acting on $\C^n$ by coordinatewise rotations.
This more restrictive condition appeared back in 
\cite[Def.~5.10 and \S 6]{karshon-grossberg} under the name ``locally toric''.
The more general notion of ``locally standard $T$-action'' that we use
was introduced by Wiemeler \cite{Wi:classification}.

Our notion of a ``unimodular labelling'',
by which we ``decorate'' $M/T$,
is equivalent to the toric topology notion of a 
``characteristic function'' \cite{DJ}.
For locally standard $T$-manifolds $M$
whose quotient map $M \to M/T$ admits a continuous section, 
Wiemeler \cite{Wi:classification} shows that 
the ``decorated'' manifold-with-corners $M/T$
determines $M$ up to equivariant diffeomorphism.
This gives the ``uniqueness'' part of the classification.
Our work extends Wiemeler's in several ways:
we do not assume the existence of a section, 
we prove the ``existence'' part of the classification,
and we obtain functors between the relevant categories. 
For more details, see \Cref{rk:wiemeler-details}.
Our work also confirms that Wiemeler's definition of a \emph{regular section},
on which his work relies, is well-posed.  See \Cref{rk:wiemeler}.

Wiemeler's approach can be taken further:
it is possible to classify locally standard $T$ manifolds $M$
that do not admit a section $M/T \to M$
by keeping track of their \emph{local} regular sections.
With this approach, rather than relying 
on the classification of principal $T$-bundles 
through their Chern classes as we do in this paper,
one repeats the steps of the proof of the classification
of principal $T$-bundles while adapting them
to the more general case of locally standard $T$-manifolds.

We now give an outline of this paper.

In \Cref{sec:example}, to illustrate a usage of 
our general classification theorem, we use it to explicitly classify 
locally standard $T$ manifolds $M$ whose quotient $M/T$ is one dimensional.

In \Cref{sec:defs},
we give precise definitions and statements,
leading to a category-theoretical formulation of our main theorem.
In \Cref{sec:differential spaces}, we give background
on differential spaces in the sense of Sikorski.
In \Cref{sec:functions descend}, we show that an equivariant diffeomorphism 
between models for locally standard $T$-manifolds
descends to a diffeomorphisms of manifolds-with-corners.
In \Cref{sec:invariants}, we prove the lemmas from \Cref{sec:defs}
that we needed for stating our main theorem.

In \Cref{sec:proof modulo cutting}, 
we present the category theoretical arguments that will imply 
our classification, assuming the existence and properties
of two functors:
the \emph{cutting} functor\footnote{Adapted from Lerman's
\emph{symplectic cutting} \cite{L}.},
which takes principal $T$-bundles over ``decorated''
manifolds-with-corners to locally standard $T$-manifolds,   
and the \emph{simultaneous toric radial blowup} functor,
which goes the other way.
In Sections \ref{sec:cut-1}--\ref{sec:cutting}
and \ref{sec:blowups}--\ref{sec:blowup smoothly}
we establish the existence and properties of these two functors.

Finally, in \Cref{sec:literature}, 
we comment on other related results in the literature.

\subsection*{Acknowledgements}
We are grateful to Michael Wiemeler, Michael Davis,
and Soumen Sarkar, for helpful discussions. 

This research was partially funded 
by the JSPS Strategic Young Researcher Overseas Visits Program 
for Accelerating Brain Circulation
``Deepening and Evolution of Mathematics and Physics, 
building of International Network Hub based on OCAMI";
by JSPS KAKENHI Grant Numbers 15K17531, 17K14196 and 21K03262;
and by the Natural Sciences and Engineering Research Council of Canada.

%===========================================
%\section{Applications}
\section{Examples}
%===========================================
\labell{sec:example}

\subsection{Principal bundles}

Every principal $T$-bundle $\Pi \colon M \to B$ 
is a locally standard $T$-manifold.
A locally standard $T$-manifold $M$ is a principal $T$-bundle
if and only if its quotient $M/T$ is a manifold (without boundary).
In this case, the unimodular labelling is empty,
and we obtain the usual classification of principal $T$-bundles
by their Chern classes.

\subsection{One dimensional quotients}

Let $\T^d := (S^1)^d$,
and let $e_1,\ldots,e_d$ be the standard generators of its weight lattice.

We have the following examples of locally standard $\T^d$-manifolds
with one dimensional quotients.
In all these examples, because the quotient $M/T$ is one dimensional, 
the Chern class $c$ is automatically trivial.

\begin{enumerate}

\item
Take the action of the torus $T=\T^d$ on the manifold $M=S^1 \times \T^d$
by multiplication on the second factor.
Then $M/T \cong S^1$, and the unimodular labelling is empty.

\item
Take the action of the torus $T=\T^d$ on the manifold $M=\R \times \T^d$ 
by multiplication on the second factor.
Then $M/T \cong \R$, and the unimodular labelling is empty.

\item
Take the action of the torus $T=S^1 \times \T^{d-1} \cong \T^d$
on the manifold $M=\C \times \T^{d-1}$ by
$$ (t_1,\mathbf{t}) \colon (z,\mathbf{s}) \mapsto (t_1z,\mathbf{ts}).$$
Then $M/T \cong [0,\infty)$,
and the unimodular labelling takes its boundary point to $\pm e_1$.

\item
Consider the sphere $S^2 \subset \C \times \R$.
Take the action of the torus $T=S^1 \times \T^{d-1} \cong \T^d$
on the manifold $M=S^2 \times \T^{d-1}$ by
$$ (t,\mathbf{t}) \colon ((z,h),\mathbf{s}) \mapsto ((tz,h),\mathbf{ts}).$$
Then $M/T \cong [-1,1]$, and the unimodular labelling takes 
both of its boundary points to $\pm e_1$.

\item
Consider the sphere $S^3 \subset \C^2$.
Consider the action of the torus $T=\T^2 \times \T^{d-2} \cong \T^d$
on the manifold $M=S^3 \times \T^{d-2}$ by
$$ ((t_1,t_2),\mathbf{t}) \colon ((z_1,z_2),\mathbf{s}) 
   \mapsto ((t_1z_1,t_2z_2),\mathbf{ts}).$$
Then $M/T \cong [-1,1]$ and the unimodular labelling 
takes its boundary points to $\pm e_1$ and $\pm e_2$.

\item
Let $k$ be an integer $\geq 2$. 
Let $w \in \{1,\ldots,k\}$.  Suppose that $\gcd(w,k) = 1$.
The group $C_k$ of $k$th roots of unity acts on the sphere $S^3 \subset \C^2$ 
by 
$$ \xi \colon (z_1,z_2) \mapsto (\xi^w z_1, \xi z_2).$$
The \emph{lens space} $L(k;w)$ is the quotient $S^3/C_k$.
Take the action of $\T^2 \times \T^{d-2} \cong \T^d$ on the manifold
$$ L(k;w) \times \T^{d-2} $$
by
$$ ((\tau_1,\tau_2),\mathbf{t}) \colon ( [z_1,z_2], \mathbf{s} ) \mapsto 
    ( [t_1 z_1,t_2 z_2], \mathbf{ts} )$$
where $t_2^k = \tau_2$ and $t_1 = \tau_1 t_2^w$.
Then $M/T \cong [-1,1]$, and the unimodular labelling
takes one boundary point to $\pm e_1$ and the other to 
$\pm (-we_1 + ke_2)$. 
\end{enumerate}

\begin{theorem}
Let $T$ be a torus, let $M$ be a connected locally standard $T$-manifold, 
and assume that $M/T$ is one dimensional.  Then there exists 
an isomorphism $\rho \colon T \cong \T^d$ and a $\rho$-equivariant
diffeomorphism from $M$ to exactly one of the above $\T^d$-manifolds.
\end{theorem}

\begin{proof}
This follows from \Cref{thm:main0} together with the following facts.
First, by the classification of one-manifolds with boundary,
$M/T$ is diffeomorphic to either $S^1$, or $\R$, or $[0,\infty)$,
or $[-1,1]$.  Second, for each primitive element $\alpha$ 
of the weight lattice $\tZ$,
there exists an isomorphism $\tZ \cong \Z^n$ that takes $\alpha$ to $e_1$.
Third, for each two primitive elements $\alpha_1$ and $\alpha_2$ 
of the weight lattice $\tZ$, exactly one the following possibilities holds.
\begin{itemize}
\item
$\pm \alpha_1 = \pm \alpha_2$, 
and there exists an isomorphism $\tZ \cong \Z^n$ 
that takes $\alpha_1$ to $e_1$.
\item
$\{\alpha_1,\alpha_2\}$ generate the sub-lattice
$\tZ \cap \on{span}\{\alpha_1,\alpha_2\}$ of $\tZ$, 
and there exists an isomorphism $\tZ \cong \Z^n$
that takes $\alpha_1$ to $e_1$ and $\alpha_2$ to $e_2$.
\item
$\{\alpha_1,\alpha_2\}$ generate a sub-lattice 
of index $k \geq 2$ in $\tZ \cap \on{span}\{\alpha_1,\alpha_2\}$,
and there exists an isomorphism $\tZ \cong \Z^n$
that takes $\alpha_1$ to $e_1$ and $\alpha_2$ to $-we_1+ke_2$
where $1 \leq w \leq k$ and $\gcd(w,k)=1$.
\end{itemize}

\end{proof}

%===========================================
\section{Definitions and statements}
%===========================================
\labell{sec:defs}

In this section we give precise statements of our main results.
%This section contains a precise statement of our classification.

We begin with a reminder on weights of a torus, leading to the notion 
of a locally standard $T$-manifold.

A \textbf{torus} is a Lie group $T$
that is isomorphic to $\bm{\T^d} := (S^1)^d$, for some positive integer~$d$.
For a torus $T$ with Lie algebra $\algt$,
the elements of its \textbf{integral lattice}
$\tZ := \ker (\exp \colon \algt \to T)$ parametrize the 
homomorphisms from $S^1$ to $T$.
Explicitly,
the element $\eta \in \tZ$ encodes the homomorphism
$\rho_\eta \colon S^1 \to T$ that is given by
$\rho_\eta(e^{2\pi i s}) := \exp(s\eta)$.
The subset $(\algt_\Z)_{\primitive}$ of primitive lattice elements
corresponds to those homomorphisms $\rho_\eta$ that are one-to-one.
The elements of 
\begin{equation} \labell{tZhat}
 \tZhat := (\tZ)_{\primitive}/\pm 1 
\end{equation}
parametrize the (unoriented, unparametrized) circle subgroups of $T$.
We write this parametrization as
$$ \etahat \mapsto S^1_\etahat \quad \text{ for } \ 
   \etahat = \pm \eta \in \tZhat. $$

The \textbf{weight lattice} is the dual lattice, 
$\tZ^* \subset \algt^*$.
Its elements parametrize the homomorphisms from $T$ to $S^1$.
Explicitly, the element $\alpha \in \tZ^*$ encodes the homomorphism
$\exp(\eta) \mapsto e^{2 \pi i \left< \alpha , \eta \right> }$,
which we also write as $t \mapsto t^\alpha$.
%\label{p:t^alpha}
The splittings of $T$ are given by the isomorphisms
$$ \rho_{\alpha_1,\ldots,\alpha_d} \colon T \to \T^d \ , \quad
t \mapsto (t^{\alpha_1},\ldots,t^{\alpha_d}) $$
where $\alpha_1,\ldots,\alpha_d$ is a basis of the lattice $\tZ^*$.

\begin{remark} \labell{stab etaj}
For each $j$, the inverse isomorphism
$ \rho_{\alpha_1,\ldots,\alpha_d}^{-1} \colon \T^d \to T $
takes the $j$th factor of $\T^d$
to the subcircle of $T$ that is generated by $\eta_j$,
where $\eta_1,\ldots,\eta_d$ is the dual basis to $\alpha_1,\ldots,\alpha_d$.
\eor
\end{remark}

The \textbf{real weights} are the elements of $\tZ^*/{\pm 1}$.
For each weight $\alpha \in \tZ^*$, 
we denote the corresponding real weight by $\hat{\alpha}$.
For weights $\alpha_1,\ldots,\alpha_m \in \tZ^*$,
we denote by $\C^m_{\alpha_1,\ldots,\alpha_m}$ the vector space $\C^m$
with the $T$-action 
$t \cdot (z_1,\ldots,z_m) = (t^{\alpha_1}z_1, \ldots,t^{\alpha_m}z_m) $.
As real $T$-representations,
$\C^m_{\alpha_1,\ldots,\alpha_m}$ and $\C^m_{\beta_1,\ldots,\beta_m}$ 
are isomorphic if and only if there exists a permutation $\sigma$ such that 
$\hat{\alpha}_j = \hat{\beta}_{\sigma(j)}$ for all $j$.
For any real $T$-representation $V$ with fixed subspace $V^T$,
the quotient $V/V^T$ is isomorphic to some $\C^m_{\alpha_1,\ldots,\alpha_m}$.
The \textbf{non-zero real weights of $V$} are the elements
$\hat{\alpha}_1 \ldots, \hat{\alpha}_m$ of $\tZ^*/{\pm 1}$.
Whether they are a basis to the weight lattice is independent of the
choice of representatives $\alpha_1,\ldots,\alpha_m$.
This holds if and only if $m=d$ and there exists a splitting
$\rho \colon T \to \T^d$ 
and a $\rho$-equivariant real-linear isomorphism $V/V^T \to \C^d$.

We now fix a $d$-dimensional torus $T$,
and we will apply the above notions to sub-tori of~$T$.

\begin{definition} 
\labell{def:locally standard}
A {$\bm T$}\textbf{-manifold} is a (smooth) manifold $M$ equipped with 
a (smooth) $T$-action.
The \textbf{stabilizer} of a point $x \in M$
is the subgroup $H := \{ h \in T \ | \ h \cdot x = x \}$.
If $H$ is connected, the \textbf{non-zero isotropy weights} at $x$ 
are the non-zero real weights of the isotropy representation at $x$,
namely, of the linearized $H$-action on $T_xM$.
A $T$-manifold $M$ is \textbf{locally standard}
if, at each point, the stabilizer is a sub-torus $H$ and the non-zero 
isotropy weights are a basis to the weight lattice of $H$.
\end{definition}

\begin{definition} 
\labell{def:Tchart}
For any non-negative integers $n,l,m$ with $n+l = d$,
we equip the space $\C^n \times \T^{l} \times \R^m$
with the standard $\T^d$-action on its first $d$ coordinates.
A \textbf{locally standard $\mathbf{T}$-chart} on a $T$-manifold $M$
is a $\rho_{\alpha_1,\ldots,\alpha_d}$-equivariant diffeomorphism
$$\begin{array}{ccc}
 U_M & \xrightarrow{\quad\phi\quad} & \Omega \\
 \actsonup & & \actsonup \\
 T & \xrightarrow{\rho_{\alpha_1,\ldots,\alpha_d}} & \T^d
\end{array} $$
from a $T$-invariant open subset $U_M$ of $M$
to a $\T^d$-invariant open subset $\Omega$ 
of $\C^n \times \T^l \times \R^m$,
for some non-negative integers $n$ and $l$ such that $n+l=d$
and for some basis $\alpha_1,\ldots,\alpha_{d}$ 
of the weight lattice $\tZ^*$.
Note that $\phi$ determines $\alpha_1,\ldots,\alpha_d$.
\end{definition} 

\begin{lemma} \labell{locally standard criterion}
A $T$-manifold is locally standard if and only if for every point 
there exists a locally standard $T$-chart whose domain contains the point.
\end{lemma}

\begin{proof}
This is a consequence of Koszul's slice theorem \cite{koszul}.
\end{proof}

\begin{remark} \labell{lem:criterion}
\Cref{locally standard criterion} can be rephrased
as saying that our definition of ``locally standard''
in \Cref{def:locally standard}
is equivalent to Wiemeler's definition of ``locally standard''
in \cite[\S 1]{Wi:classification}.
\eor
\end{remark}

\begin{example} \labell{locally standard examples}
Principal $T$-bundles are those locally standard $T$-manifold
that are locally modelled on $\T^l \times \R^m$
(that is, $n=0$ and $\dim T = l$).
Locally toric $T$-manifolds are those locally standard $T$-manifolds
that are locally modelled on $\C^n$ (that is, $l=m=0$ and $\dim T = n$).
\eoe
\end{example}

Here is a reminder on manifolds-with-corners.  For more details,
see John Lee's textbook \cite[Section~16]{Le} 
and Douady's original paper \cite{douady}.

A map $f \colon \calO \to \calO'$ 
from an (arbitrary) subset $\calO$ of a Cartesian space $\R^N$
to an (arbitrary) subset $\calO'$ of a Cartesian space $\R^{N'}$
 --- in particular from a (relatively) open subset of $\Rplus^n \times \R^m$
to a (relatively) open subset of $\Rplus^{n'} \times \R^{m'}$ ---
is \textbf{smooth} if for each point of $\calO$ 
there exist an open neighbourhood $W$ of the point in $\R^N$
and a smooth map $F \colon W \to \R^{N'}$
such that $F|_{\calO \cap W} = f|_{\calO \cap W}$.
Such a map $f$ is a \textbf{diffeomorphism} if it is smooth 
and has a smooth inverse. 

A \textbf{chart-with-corners} on a topological space $Q$
is a homeomorphism from an open subset $U$ of $Q$
to a (relatively) open subset $\calO$ of $\Rplus^n \times \R^m$ 
for some non-negative integers $n$ and~$m$.
Two charts-with-corners 
$\varphi \colon U \to \calO \subset \Rplus^n \times \R^m$
and $\varphi' \colon U' \to \calO' \subset \Rplus^{n'} \times \R^{m'}$ 
are \textbf{compatible}
if the map $\varphi' \circ {\varphi}^{-1} \colon \varphi(U \cap U') 
\to \varphi'(U \cap U')$,
which is a homeomorphism 
from a (relatively) open subset of $\Rplus^n \times \R^m$
to a (relatively) open subset of $\Rplus^{n'} \times \R^{m'}$,
is a diffeomorphism.
An \textbf{atlas-with-corners} on $Q$ is a set of charts-with-corners
whose domains cover $Q$ and that are pairwise compatible.  
A \textbf{manifold-with-corners} is a Hausdorff second countable 
topological space $Q$ equipped with a maximal atlas-with-corners.
On a manifold-with-corners $Q$, whenever we refer to a chart-with-corners
on $Q$,
we implicitly assume that it is in the given maximal atlas-with-corners.

In a manifold-with-corners $Q$,
the \textbf{depth} of a point $q$ is the unique integer $k$
such that, for every chart $\psi$ whose domain contains $q$
and whose image is open in $\Rplus^n \times \R^m$,
exactly $k$ among the first $n$ coordinates of $\psi(q)$ vanish.
Following Douady \cite{douady},
we denote the set of points of depth $k$ by $\bm{\mathring{\del^k}Q}$.
The \textbf{interior} of $Q$ is the set $\mathring{\del^0}Q$ 
of points of depth zero; we also denote it $\bm{\mathring{Q}}$.
Its inclusion map into $Q$ is a homotopy equivalence.

A map $\psi \colon Q_1 \to Q_2$ of manifolds-with-corners
is \textbf{smooth}\footnote{Some authors use a more restrictive definition
of ``smooth map'' between manifolds-with-corners. 
For diffeomorphisms, the definitions agree.} 
if for each chart-with-corners $\varphi_1$ of $Q_1$
and $\varphi_2$ of $Q_2$, the composition
$\varphi_1 \circ \psi \circ \varphi_2^{-1}$
is a smooth map between subsets of Cartesian spaces.
This holds if and only if for each smooth function $f \colon Q_2 \to \R$
the composition $f \circ \psi \colon Q_1 \to \R$ is smooth.
The map $\psi$ is a \textbf{diffeomorphism} if it is a bijection
and $\psi$ and $\psi^{-1}$ are smooth.
More generally, $\psi$ is a \textbf{local diffeomorphism}
if for each $x \in Q_1$ it restricts to a diffeomorphism
from some neighbourhood of $x$ in $Q_1$ 
to some neighbourhood of $\psi(x)$ in~$Q_2$.
This implies that, for each $k$,
we have $\mathring{\del}^kQ_1 = \psi^{-1}(\mathring{\del}^kQ_2)$.

Recall that we have fixed a $d$-dimensional torus $T$.

\begin{definition} \labell{def:unimodular} \ 
Fix a manifold-with-corners $Q$. 
A \textbf{unimodular labelling} on $Q$ is a function 
$$\lambdahat \colon \mathring{\del^1}Q \to \tZhat$$
(see \eqref{tZhat}) that has the following property.
For each point in $Q$ there exist 
non-negative integers $n$ and $m$ with $0 \leq n \leq d$,
and a basis $\eta_1,\ldots\eta_d$ of the integral lattice $\tZ$,
and a chart-with-corners 
$\psi \colon U_Q \to \calO \subset \Rplus^n \times \R^m$
whose domain contains the point, such that 
for each $x \in\mathring{\del}^1 Q \cap U_Q$ and $j \in\{1,\ldots,n\}$,
if the $j$th coordinate of $\psi(x)$ vanishes,
then $\lambdahat(x) = \pm \eta_j$.
Such a pair $(\psi,(\eta_1,\ldots,\eta_d))$
is called a \textbf{unimodular chart} on $(Q,\lambdahat)$.
\eod
\end{definition}

\begin{remark} \labell{rk:mfld-w-faces}
A manifold-with-corners $Q$ is a \textbf{manifold with faces} 
(see J\"anich \cite{janich})
if, for each $k$, each point in $\mathring{\del^k} Q$ is in the closure
of exactly $k$ connected components of $\mathring{\del^1} Q$.
The \textbf{codimension $\bm{k}$ faces} 
are then the closures of the connected components of $\mathring{\del^k} Q$.
The \textbf{facets} are the codimension $1$ faces.
If there exists a unimodular labelling 
$\lambdahat \colon {\mathring\del^1}Q \to \tZhat$,
then $Q$ is a manifold with faces,
and the unimodular labelling can be identified with a function
from the set of facets to $\tZhat$.
\eor
\end{remark}

\begin{remark} \labell{rk:Tbundles}
We briefly recall the classification of principal $T$-bundles
over manifolds-with-corners.
For principal circle bundles over manifolds,
relevant details can be found in Kostant's paper \cite{Ko}.
The case of principal $T$-bundles over manifolds-with-corners
is analogous.

First, one defines the Chern class of a principal $T$-bundle
over a manifold or manifold-with-corners $Q$.
This is a cohomology class on $Q$ of degree 2
with coefficients in the integral lattice $\algt_\Z$ of $T$. 
Its definition is through a choice of local trivializations
whose domains cover~$Q$; the resulting (\v{C}ech) cohomology class
is independent of this choice.
Second, one shows that the Chern class determines the bundle
up to isomorphism; this is the \emph{uniqueness} part of the classification.
Third, one shows that every cohomology class on $Q$ of degree 2
with coefficients in $\algt_\Z$ can be obtained in this way; 
this is the \emph{existence} part of the classification.

Given principal $T$-bundles $P \to Q$ and $P' \to Q'$,
every equivariant smooth map $P \to P'$ 
descends to a smooth map $Q \to Q'$ 
that intertwines the Chern classes.
Moreover, a smooth map $Q \to Q'$ 
admits a smooth equivariant lifting $P \to P'$
if and only if it intertwines the Chern classes. 

Thus, we obtain a functor
from the category of principal $T$-bundles over manifolds-with-corners,
where morphisms are equivariant smooth maps,
to the category of manifolds-with-corners $Q$
that are decorated with a cohomology class in $H^2(Q;\algt_\Z)$,
where morphisms are
smooth maps that intertwine the cohomology classes.
By ``existence'', this functor is essentially surjective.
By the above criterion for lifting, 
the functor is full.

Moreover, $P \to P'$ is a diffeomorphism iff $Q \to Q'$ is a diffeomorphism,
and $P \to P'$ is a local diffeomorphism iff $Q \to Q'$ is a local 
diffeomorphism.  So the functor restricts to functors 
between the subcategories with the same objects
where the morphisms are diffeomorphisms or local diffeomorphisms
rather than more general smooth maps. 
All these functors are essentially surjective and full.
\eor
\end{remark}

The following three lemmas give the invariants
of locally standard $T$-actions.

\begin{lemma} \labell{l:quotient}
Let $M$ be a locally standard $T$-manifold.
Then there exists a unique manifold-with-corners structure on $M/T$
with which a real valued function on $M/T$ is smooth
if and only if its pullback to $M$ is smooth.
\end{lemma}

\begin{lemma} \labell{l:labelling}
Let $M$ be a locally standard $T$-manifold.
Equip $Q := M/T$ with the manifold-with-corners structure
that is described in Lemma~\ref{l:quotient}.
Then the preimage in~$M$ of each point of $\mathring{\del^1}Q$
is an orbit whose stabilizer is a circle subgroup of $T$,
and the map 
$$ \hat\lambda_M \colon \mathring{\del^1}Q \to \tZhat $$
that associates to each point in $\mathring{\del^1}Q$
the element of $\tZhat$ that encodes the stabilizer of its preimage
is a unimodular labelling.
\end{lemma}

\begin{lemma} \labell{l:cM}
Let $M$ be a locally standard $T$-manifold.
Let $\Mfree$ be the set of points of $M$ whose $T$ stabilizer is trivial.
Then $\Mfree$ is the preimage in $M$
of the interior $\intQ$ of the manifold-with-corners $Q := M/T$,
the map $\pi|_{\Mfree} \colon \Mfree \to \intQ$ is a principal $T$-bundle,
and 
there exists a unique cohomology class $c_M \in H^2(Q;\tZ)$
whose pullback to $H^2(\intQ;\tZ)$ is the Chern class 
of this principal $T$-bundle.
\end{lemma}

We prove Lemmas \ref{l:quotient}, \ref{l:labelling} and \ref{l:cM} 
in Section \ref{sec:invariants}.  Assuming these lemmas,
we proceed to define the invariants of a locally standard $T$-manifold
and to state our main theorems.

\begin{definition} \labell{def:invariants}
To a locally standard $T$-manifold $M$, we associate the following items.
\begin{enumerate}
\item \label{def:Q}
The quotient $Q := M/T$, equipped with its manifold-with-corners structure 
that is described in Lemma~\ref{l:quotient};
we call it the \textbf{quotient of the $\bm T$-manifold $\bm M$}.
\item \labell{def:lambda}
The unimodular labelling $\hat{\lambda} \colon \mathring{\del^1}Q \to \tZhat$
that is described in Lemma~\ref{l:labelling};
we call it \textbf{the unimodular labelling of the $\bm T$-manifold $\bm M$}.
\item \labell{def:C}
The cohomology class $c_M \in H^2(Q;\tZ)$ 
that is described in Lemma~\ref{l:cM};
we call it the \textbf{Chern class of the $\bm T$-manifold $\bm M$}.
\end{enumerate}
\end{definition}

Parts \eqref{main:descends}, \eqref{main:lifts}, and \eqref{main:exists}
of the following theorem 
contain the classification that we sketched in Section~\ref{sec:intro}.

\begin{theorem}\labell{thm:main} \ 
\begin{enumerate}
\item \labell{main1}
Let $M_1$ and $M_2$ be locally standard $T$-manifolds,
let $Q_1 := M_1/T$ and $Q_2 := M_2/T$
and $\pi_1 \colon M_1 \to Q_1$ and $\pi_2 \colon M_2 \to Q_2$
be their quotients and their quotient maps,
let $\lambdahat_{1}$ and $\lambdahat_{2}$
be their unimodular labellings, 
and let $c_{M_1}$ and $c_{M_2}$ be their Chern classes.

\begin{enumerate} 
\item \labell{main:descends}
Let $\tilde{\psi} \colon M_1 \to M_2$ be an equivariant diffeomorphism.
Then there exists a unique diffeomorphism $\psi \colon Q_1 \to Q_2$\
such that the diagram \eqref{below} commutes.
The diffeomorphism $\psi$ satisfies
$\psi^* c_{M_2} = c_{M_1}$ and $\lambdahat_{2} \circ \psi = \lambdahat_{1}$.

\item \labell{main:lifts}
Let $\psi \colon Q_1 \to Q_2$ be a diffeomorphism.
Suppose that $\psi^* c_{M_2} = c_{M_1}$
and that $\lambdahat_{2} \circ \psi = \lambdahat_{1}$.
Then there exists a $T$-equivariant diffeomorphism 
$\tpsi \colon M_1 \to M_2$ such that the diagram \eqref{below} commutes.
\end{enumerate}

\begin{equation} \labell{below}
\xymatrix{
 M_1 \ar[r]^{\tpsi} \ar[d]^{\pi_1} & M_2 \ar[d]^{\pi_2} \\
 Q_1 \ar[r]^{\psi} & Q_2
}
\end{equation}

\item \labell{main:exists}
Let $Q$ be a manifold-with-corners,
let $\lambdahat_Q \colon \mathring{\del^1} Q \to \tZhat$
be a unimodular labelling,
and let $c_Q$ be a cohomology class in $H^2(Q;\tZ)$.
Then there exist a locally standard $T$-manifold $M$
and a diffeomorphism of manifolds-with-corners $\psi \colon M/T \to Q$
such that $\psi^* c_Q = c_M$
and such that $\lambdahat_Q \circ \psi = \lambdahat_M$.

\end{enumerate}
\end{theorem}

\begin{proof}[Proof of Part~\eqref{main:descends} of \Cref{thm:main}] 
Because $\tilde{\psi}$ is $T$-equivariant, there exists a unique map
$\psi \colon Q_1 \to Q_2$ such that the diagram \eqref{below} commutes.
Because $\tilde{\psi}$ is a $T$-equivariant diffeomorphism
and by \Cref{l:quotient} and Part~\eqref{def:Q} of \Cref{def:invariants},
$\psi$ is a diffeomorphism of manifolds-with-corners.
Because $\tpsi$ is a $T$-equivariant diffeomorphism 
and by \Cref{l:labelling} and Part~\eqref{def:lambda} of \Cref{def:invariants},
we have ${\lambdahat_{2} \circ \psi = \lambdahat_{1}}$.
Because $\tilde{\psi}$ is a $T$-equivariant diffeomorphism
and by \Cref{l:cM} and Part~\eqref{def:C} of \Cref{def:invariants},
we have $\psi^* c_{M_2} = c_{M_1}$. 
\end{proof}

\begin{proof}[Pointers to proofs of Parts \eqref{main:lifts}
and~\eqref{main:exists} of \Cref{thm:main}.] \ %\linebreak
In \Cref{rk:functor} we rephrase these parts of \Cref{thm:main} 
as saying that a certain functor is full and essentially surjective.
In \Cref{what we need 1} and \Cref{what we need 2}
we prove that this functor is full and essentially surjective,
assuming the existence and properties of some additional functors,
which are formulated in Proposition \ref{natural isos}.
In Sections \ref{sec:cut-1}--\ref{sec:cutting}
and \ref{sec:blowups}--\ref{sec:blowup smoothly}
we establish the existence and properties of these additional functors.
\end{proof}

We now formulate our classification in a category-theoretical terms.
Here we consider categories whose morphisms are diffeomorphism.
We expect to be able to allow more general morphisms; 
see \Cref{rk:transverse maps}.

A \textbf{decorated manifold-with-corners} 
is a triple $(Q,\lambdahat,c)$, where $Q$ is a manifold-with-corners,
$\lambdahat \colon \del^1  Q \to \tZhat$ is a unimodular labelling,
and $c$ is a cohomology class in $H^2(Q,\tZ)$.
Sometimes we abbreviate and write $Q$ instead of $(Q,\lambdahat,c)$.

\begin{definition} \labell{def:functor}
Let $\frakM_\iso$ denote the category whose objects are 
locally standard $T$-manifolds
and whose morphisms are $T$-equivariant diffeomorphisms.
Let $\frakQ_\iso$ denote the category whose objects are 
decorated manifolds-with-corners $(Q,\lambdahat,c)$
and whose morphisms are diffeomorphism of the manifolds-with-corners
that intertwine their unimodular labellings and their homology classes;
we call these \emph{decoration-preserving diffeomorphisms}.
Part~\eqref{main:descends} of \Cref{thm:main} defines a functor
($M \mapsto Q := M/T$)
from the category $\frakM_\iso$ to the category $\frakQ_\iso$.
We call it the \textbf{quotient functor}.
\end{definition}

\begin{remark} \labell{rk:functor}
The quotient functor $\frakM_\iso \to \frakQ_\iso$
gives rise to a map from the set of isomorphism classes
of the category $\frakM_\iso$ to the set of isomorphism classes
of the category $\frakQ_\iso$.
Part~\eqref{main:lifts} of~\Cref{thm:main} can be rephrased
as saying that the quotient functor 
$\frakM_\iso \to \frakQ_\iso$ is full.
This implies that the map on isomorphism classes is one-to-one.
Part~\eqref{main:exists} of~\Cref{thm:main}
can be rephrased as saying that the quotient functor 
$\frakM_\iso \to \frakQ_\iso$ is essentially surjective.
This means that the map on isomorphism classes is onto.
\eor
\end{remark}

\begin{remark} \labell{rk:1c}
The functor $\frakM_\iso \to \frakQ_\iso$ is not faithful, but 
we know the extent to which it's not faithful: 
two equivariant diffeomorphisms 
$\psi, \psi' \colon M_1 \to M_2$
descend to the same diffeomorphism $Q_1 \to Q_2$ if and only if 
there exists a smooth map $t \colon Q_1 \to T$
such that $\psi'(x) = t(\pi_1(x)) \cdot \psi(x)$
for all $x \in M_1$.
This follows from the analogous fact for principal $T$ bundles, 
together with the equivalence of categories of \Cref{equiv of cat}.
\eor
\end{remark}

% --------------------------------------------------------
\section{Differential spaces}
% --------------------------------------------------------
\labell{sec:differential spaces}

It is convenient to work with the notion of a differential space,
as axiomatized by Sikorski.
In this section we give some relevant definitions and facts.
For some details, see \cite{sikorski1,sikorski2,sniat}. 

\begin{definition} \labell{diff space}
A \textbf{differential space} 
is a topological space $X$
that is equipped with \textbf{differential structure},
which is a nonempty collection $\calF$ of real valued functions on $X$
that satisfies the following axioms.
\begin{enumerate}[topsep=0pt]
\item
The given topology on $X$ coincides
with the initial topology that is induced from~$\calF$.
\item
If $g_1,\ldots,g_k$ are in $\calF$
and $h \colon \R^k \to \R$ is smooth then $h(g_1,\ldots,g_k)$ is in $\calF$.
\item
Given any real valued function $g$ on $X$,
if every point in $X$ has a neighbourhood on which $g$ coincides
with a function in $\calF$, then $g$ is in $\calF$.
\end{enumerate}
\end{definition}

Given differential spaces $(X,\calF_X)$ and $(Y,\calF_Y)$,
a map from $X$ to $Y$ is \textbf{smooth}
if its composition with each element of $\calF_Y$ is an element of $\calF_X$.
It is a \textbf{diffeomorphism} if it is smooth and has a smooth inverse.
Every smooth map is continuous; every diffeomorphism is a homeomorphism.
A map from $X$ to $\R$ is smooth,
with $\R$ equipped with its standard differential structure,
if and only if the map is in $\calF_X$.

\begin{remark}\labell{topology}
Let $X$ be a topological space, and let $\calF$ be a nonempty collection
of real valued functions on $X$.
Suppose that (2) holds.  Then (1) is equivalent to (1'):
\begin{enumerate}[topsep=0pt]
\item[(1')]
The functions in $\calF$ are continuous,
and the topology on $X$ is $\calF$-regular. 
\end{enumerate}
(Being $\calF$-regular means that for any closed subset $C$ of $X$ 
and any point $x \in X$ outside $C$ 
there exists $\rho \in \calF$ that is supported in $X \setminus C$
and is identically $1$ on a neighbourhood of $x$.)
Indeed, let $\tau_X$ be the given topology on $X$
and let $\tau_\calF$ be the initial topology that is induced from $\calF$.
Then $\tau_\calF \subset \tau_X$
if and only if the functions in $\calF$ are continuous,
and, assuming that this holds, $\tau_\calF \supset \tau_X$
if and only if $\tau_X$ is $\calF$-regular (see \cite[Lemma~2.15]{flows}).
\eor
\end{remark}

\begin{remark}\labell{subspace}
On a subset $A$ of a differential space $X$, equipped with the subset topology,
the \textbf{subspace differential structure}
consists of those real valued functions $g$ 
such that for every point in $A$ there exist a neighbourhood $U$ in $X$
and a smooth function on $X$ that coincides with $g$ on $A \cap U$.
For subsets $B \subset A \subset X$, 
the differential structure on $B$ as a subset of $A$
coincides with the differential structure on $B$ as a subset of $X$.
Given another differential space $Y$, a map $Y \to A$ is smooth 
if and only if the composition 
$Y \to A \xrightarrow{\text{inclusion}} X$
is smooth.
%\ynote{
%We refer to \Cref{subspace} in the proof of \Cref{thetabar on subset}.}
\eor
\end{remark}

\begin{remark}
Given a differential space $(X,\calF_X)$,
by associating to each open subset $U$ its subspace differential structure
$\calF_U$, we obtain a sheaf of continuous real-valued functions
that contains the constant functions.
Such a sheaf is the same thing as a \emph{functional structure},
introduced by Hochschild in \cite[Chap.~VI]{Hoch}
and used by Bredon in \cite[Chap.~VI, p.~297]{Br}.
A functional structure $U \mapsto \calF_U$ 
on a topological space $X$ arises in this way if and only if 
each $\calF_U$ it is closed under compositions with smooth functions
and the given topology on $X$
coincides with the initial topology induced by $\calF_X$.
\eor
\end{remark}

\begin{remark} \labell{mfld}
On a (smooth) manifold or manifold-with-boundary or manifold-with-corners,
the set of all real valued smooth functions is a differential 
structure.\footnote{Condition (1') of \Cref{topology}
uses that the underlying topological space is locally compact Hausdorff.}
A map between manifolds or manifolds-with-boundary or manifolds-with-corners
is smooth in the usual sense
if and only if it is smooth as a map of differential spaces.
On a differential space that is second countable
and that is locally diffeomorphic
to open subsets of the spaces $\Rplus^n \times \R^m$ for non-negative integers 
$n$ and $m$, 
there exists a unique manifold-with-corners structure
whose differential structure coincides with the given one.
%\ynote{ \\ We refer to \Cref{mfld} in the proof of \Cref{f to g}, and 
%in the proof of \Cref{l:quotient} in \Cref{sec:invariants}.}
\eor
\end{remark}

\begin{remark} \labell{quotient}
Let $X$ be a manifold and $G$ a compact Lie group acting on $X$.
Then $X/G$, with its quotient topology, is Hausdorff and second countable.
Moreover, the set of real-valued functions on $X/G$ whose pullback to $X$
is smooth is a differential structure,
called the \textbf{quotient differential structure}.
The proof that the initial topology on $X/G$ coincides with the quotient
topology on $X/G$ uses the fact that the $G$-average of a smooth function
is smooth. 
%\ynote{ \\ We refer to \Cref{quotient} in the proof 
%of \Cref{thetabar on subset} and the statement of \Cref{structure on chart}.}
\eor
\end{remark}

\begin{remark} \labell{subquotient}
Let $X$ be a manifold and $G$ a compact Lie group acting on $X$.
Then, for any $G$-invariant subset $A$ of $X$,
the differential structure on $A/G$ as a subset of $X/G$
coincides with its differential structure as a quotient of $A$.
This too uses the fact that the $G$-average of a smooth function is smooth.
%\ynote{ \\ We refer to \Cref{subquotient} 
%in the proofs of \Cref{thetabar on subset} and of \Cref{structure on chart}.}
\eor
\end{remark}

In the following two lemmas we spell out another way
to construct a manifold structure on a topological space or on a set.

\begin{lemma} \labell{mfld lemma1}
Let $X$ and $Q$ be topological spaces.   Assume that  $Q$  is Hausdorff
and second countable.  Let 
$$\pi \colon X \to Q$$ be a continuous map.
Let $( U_i )_{i\in I}$ be an open cover of $Q$. For each $i \in I$,
let $X_i$ be a manifold-with-corners 
and let $\phi_i \colon \pi^{-1}(U_i) \to X_i$
be a homeomorphism.  Let
$$ \pi_i := \pi \circ \phi_i^{-1} \colon X_i \to U_i.$$
Suppose that, for each $i,j \in I$,
the homeomorphism $\phi_j \circ \phi_i^{-1}$
from the open subset $\pi_i^{-1}(U_i \cap U_j)$ of $X_i$
to the open subset $\pi_j^{-1}(U_i \cap U_j)$ of $X_j$ is a diffeomorphism.
Then there exists a unique manifold-with-corners structure on $X$
such that each $\phi_i$ is a diffeomorphism.

If each $X_i$ is a manifold (namely, its boundary is empty),
then $X$ is a manifold.
%\ynote{\\ We refer to \Cref{mfld lemma1} in the proofs
%of \Cref{mfld lemma2} and \Cref{cut smooth revised}.}
\end{lemma}

\begin{proof}
Let $x,y \in X$.
First, 
suppose $\pi(x) \neq \pi(y)$. 
Because $Q$ is Hausdorff, there exist disjoint open subset $U_x,U_y$ of $Q$
such that $\pi(x) \in U_x$ and $\pi(y) \in U_y$.
The preimages $W_x := \pi^{-1}(U_x)$ and $W_y := \pi^{-1}(U_y)$
are then disjoint open subsets of $X$ such that $x \in W_x$ and $y \in W_y$.
Next, 
suppose $\pi(x) = \pi(y) =:q$. Let $i \in I$ be such that $q \in U_i$.
Because $X_i$ is Hausdorff, there exist disjoint open subsets 
$W_x',W_y'$ of $X_i$
such that $\phi_i(x) \in W_x'$ and $\phi_i(y) \in W_y'$. 
The preimages $W_x := \phi_i^{-1}(W_x')$ and $W_y := \phi_i^{-1}(W_y')$
are then disjoint open subsets of $X$ such that $x \in W_x$ and $y \in W_y$.
Thus, $X$ is Hausdorff.

Let $\{B_n\}_{n \in \N}$ be a countable basis for the topology of $Q$,
such that for each $n \in \N$
there exists $i_n \in I$ such that $B_n \subset U_{i_n}$.
For each $n \in \N$, let $\{B_{n,m}\}_{m \in \N}$
be a countable basis for the topology of $X_{i_n}$.
Then $\{\phi_{i_n}^{-1}(B_{n,m})\}$ 
is a countable basis for the topology of $X$.

For each $i$, let 
$\{ \varphi_{i,j} \colon W_{i,j} \to \Omega_{i,j} \ | \ j \in J_i\}$
be an atlas-with-corners on $X_i$.  Then 
$\bigcup_i \{ \varphi_{i,j} \circ \phi_i|_{\phi_i^{-1}(W_{i,j})} \ | \ 
 j \in J_i \}$ is an atlas-with-corners on $X$
with respect to which each $\phi_i$ is a diffeomorphism,
and any maximal atlas-with-corners on $X_i$ with respect to which 
each $\phi_i$ is a diffeomorphism must contain it.

Because $\phi_i \colon \pi^{-1}(U_i) \to X_i$ are diffeomorphisms
and their domains are an open cover of~$X$,
if $X_i$ are manifolds-with-corners then so is $X$,  
and if $X_i$ are manifolds then so is~$X$.
\end{proof}

\begin{lemma} \labell{mfld lemma2}
Let $P$ be a set, and let $Q$ be a topological space
that is Hausdorff and second countable.
Let $\pi \colon P \to Q$ be a map (of sets).
Let $( U_i )_{i\in I}$ be an open cover of $Q$.
For each $i \in I$, 
let $X_i$ be a manifold, and let $\phi_i \colon \pi^{-1}(U_i) \to X_i$
be a bijection such that the composition 
$$\pi_i := \pi \circ \phi_i^{-1} \colon X_i \to U_i$$ 
is continuous.
Suppose that, for each $i,j \in I$, the bijection $\phi_j \circ \phi_i^{-1}$
from the open subset $\pi_i^{-1}(U_i \cap U_j)$ of $X_i$
to the open subset $\pi_j^{-1}(U_i \cap U_j)$ of $X_j$ is a diffeomorphism.
Then there exists a unique manifold structure on $P$
such that each $\pi_i$ is a diffeomorphism.
%\ynote{\\ We refer to \Cref{mfld lemma2} in the proof of \Cref{cor:structure}.}
\end{lemma}

\begin{proof}
By \Cref{mfld lemma1}, it is enough to show
that there exists a unique topology on $P$
such that the maps $\phi_i$ are homeomorphisms.
Indeed, let 
$$\tau_P := \{ W \subset P \ | \ \phi_i(W \cap \pi^{-1}(U_i))
\text{ is open in } X_i \text{ for each } i \in I \}. $$
Then $\tau_P$ is a topology on $P$, 
the maps $\phi_i$ are homeomorphisms with respect to this topology, 
and any topology with this property must coincide with $\tau_P$.
\end{proof}

% --------------------------------------------------------
\section{Diffeomorphisms of models descend: from $M$ to $Q$}
% --------------------------------------------------------
\labell{sec:functions descend}

In this section, we work purely within models.
In the next section,
we will use the results of this section 
to prove the lemmas from \Cref{sec:defs}
that we used to state our main theorem, \Cref{thm:main}.

For any non-negative integers $n$, $l$, and $m$, we consider
the standard $\T^{n+l}$-action on $\C^n \times \T^l \times \R^m$,
\begin{multline*}
(\lambda_1,\ldots, \lambda_{n+l}) 
  \cdot (z_1,\ldots,z_n;
   b_{1},\ldots,b_{l};x_1,\ldots,x_m) \\
 = (\lambda_1 z_1, \ldots, \lambda_n z_n;
     b_{1} \lambda_{n+1}, \ldots, b_{l} \lambda_{n+l}; x_1,\ldots, x_m),
\end{multline*}
and the \textbf{model quotient map} 
\begin{multline} \labell{theta}
\theta \colon \C^n \times \T^l \times \R^m
\to \Rplus^n \times \R^m , \\
\theta(z_1,\ldots,z_n;b_{1},\ldots,b_{l};x_1,\ldots,x_m) 
  := (|z_1|^2,\ldots,|z_n|^2;x_1,\ldots,x_m).
\end{multline}
We will use the same symbol $\theta$ for these maps 
with different values of $n,l,m$.

\begin{lemma} \labell{thetabar is homeo}
Let $n$, $l$, and $m$ be non-negative integers.  
Let $\Omega$ be an invariant open subset of $\C^n \times \T^l \times \R^m$,
and let $\calO := \theta(\Omega) \subset \Rplus^n \times \R^m$.
Then $\calO$ is (relatively) open in $\Rplus^n \times \R^m$.
%\ynote{We refer to \Cref{thetabar is homeo} 
%in the proofs of \Cref{f to g} and of \Cref{structure on chart}.}
\end{lemma}

\begin{proof}
Because $\Omega$ is open, its complement is closed;
because $\theta$ is proper, it's a closed map;
so the image of the complement of $\Omega$ is closed. 
Because $\theta$ is onto $\Rplus^n \times \R^m$
and its level sets are the $\T^d$ orbits,
and because $\Omega$ is $\T^d$-invariant,
the image of the complement of $\Omega$
is the complement of the image of $\Omega$ in $\Rplus^n \times \R^m$.
Since this complement is closed,
the image of $\Omega$ is open in $\Rplus^n \times \R^m$.
\end{proof}

\begin{lemma} \labell{thetabar is diffeo}
Let $n$, $l$, and $m$ be non-negative integers.  Then 
a function $h \colon \Rplus^n \times \R^m \to \R$ is smooth
if and only if its pullback
$h \circ \theta \colon \C^n \times \T^l \times \R^m \to \R$ is smooth.
%\ynote{ \\ We refer to \Cref{thetabar is diffeo} in the proof
%of \Cref{thetabar on subset}.}
\end{lemma}

\begin{proof}
Because $\theta$ is smooth, if $h$ is smooth then $h \circ \theta$ is smooth.
For the converse direction, suppose that $h \circ \theta$ is smooth.
Then its restriction to $\R^n \times \T^l \times \R^m$
is an even function in the first $n$ coordinates
and is independent of the next $l$ coordinates. 
By a parametrized version of a theorem of Whitney \cite{whitney},
there exists 
a smooth function $g \colon \R^n \times \R^m \to \R$
such that 
$(h\circ\theta)(y_1,\ldots,y_n;b_1,\ldots,b_l;x_1,\ldots,x_m)
 = g(y_1^2,\ldots,y_n^2;x_1,\ldots,x_m)$
for all $(y_1,\ldots,y_n;b_1,\ldots,b_l;x_1,\ldots,x_m)
 \in \R^n \times \T^l \times \R^m$.
Necessarily, $h = g\big|_{\Rplus^n \times \R^m}$.

Alternatively, because $h \circ \theta$ is $\T^{n+l}$ invariant
and by a theorem of Schwarz \cite{Sc},
$h \circ \theta$ can be expressed as a smooth function
of $\T^{n+l}$-invariant polynomials:
$h \circ \theta = g \circ (g_1,\ldots,g_k)$ 
where $g_1,\ldots,g_k$ are $\T^{n+l}$-invariant polynomials
and $g \in C^\infty(\R^k)$. 
But every $\T^{n+l}$-invariant polynomial
is a polynomial in the coordinates of $\theta$,
so $g \circ (g_1,\ldots,g_k) 
 = g \circ (q_1 \circ \theta, \ldots, q_k \circ \theta)
 = (g \circ (q_1,\ldots,q_k)) \circ \theta$
for polynomials $q_1,\ldots,q_k$ on $\R^n \times \R^m$.
This implies that $h = g \circ (q_1,\ldots,q_k)\big|_{\Rplus^n \times \R^m}$,
and so $h$ is smooth.
\end{proof}

\begin{corollary} \labell{thetabar on subset}
Let $n$, $l$, and $m$ be non-negative integers.
Let $\Omega$ be an invariant open subset of $\C^n \times \T^l \times \R^m$,
and let $\calO = \theta(\Omega)$.
Then a function $h \colon \calO \to \R$ is smooth
if and only if its pullback $h \circ \theta \colon \Omega \to \R$
is smooth.
%\ynote{ \\ We refer to \Cref{thetabar on subset} 
%in the proofs of \Cref{structure on chart} and of \Cref{f to g}.}
\end{corollary}

\begin{proof}
A rephrasing of \Cref{thetabar is diffeo} is that 
$\theta$ induces a diffeomorphism of differential spaces 
from $(\C^n \times \T^l \times \R^m)/\T^{n+l}$,
with its quotient differential structure 
induced from $\C^n \times \T^l \times \R^m$ (\Cref{quotient}), 
to $\Rplus^n \times \R^m$,
with its subspace differential structure induced from $\R^n \times \R^m$
(\Cref{subspace}).
This restricts to a diffeomorphism 
from $\Omega/\T^d$, with its subspace structure
induced from $(\C^n \times \T^l \times \R^m)/\T^{n+l}$,
to $\calO$, with its subspace structure induced from $\Rplus^n \times \R^m$.
\Cref{thetabar on subset} then follows from 
the fact that, on $\Omega/\T^d$, 
the subspace structure induced from $(\C^n \times \T^l \times \R^m)/\T^{n+l}$
coincides with the quotient differential structure induced from~$\Omega$
(\Cref{subquotient}).
\end{proof}

\begin{corollary} \labell{f to g}
Fix non-negative integers $n, n', l, l', m, m'$ such that $n+l=n'+l'=:d$.
Let $\rho \colon \T^d \to \T^d$ be an isomorphism, and let 
$$ f \colon \Omega \to \Omega' $$
be a $\rho$-equivariant smooth map
from an invariant open subset $\Omega$ of $\C^n \times \T^l \times \R^m$
to an invariant open subset $\Omega'$ 
of $\C^{n'} \times \T^{l'} \times \R^{m'}$.
Let $\calO = \theta(\Omega)$ and $\calO' = \theta(\Omega')$.
Then there exists a unique map
$g \colon \calO \to \calO'$
such that the following diagram commutes:
$$ \xymatrix{
\Omega \ar[r]^{f} \ar[d]_{\theta} & **[r] \Omega' \ar[d]^{\theta} \\
\calO \ar[r]^{g} & **[r] \calO' \ .
} $$
Moreover, the map $g$ is smooth. 
Moreover, if $f$ is a diffeomorphism, then $g$ is a diffeomorphism
%\ynote{Moreover, if $g$ is a diffeomorphism, then $f$ is a diffeomorphism.
 %\\ Moreover, $f$ is a local diffeomorphism iff $g$ is a local diffeomorphism.}
%\ynote{ \\ We refer to \Cref{f to g} from \Cref{sec:M to P}.}
\end{corollary} 

\begin{proof}
By \Cref{thetabar is homeo},
$\calO$ and $\calO'$ are open in $\Rplus^n \times \R^m$ 
and in $\Rplus^{n'} \times \R^{m'}$, respectively.
Because the map $f$ is equivariant
with respect to an isomorphism $\T^d \to \T^d$,
it takes $\T^d$-orbits to $\T^d$-orbits.
Because the maps $\theta$ are onto and their level sets are the $\T^d$-orbits,
there exists a unique map $g$ such that the diagram commutes.
For any smooth function $h \colon \calO' \to \R$, 
because the diagram commutes and $\theta \circ f$ is smooth,
the invariant function 
$(h \circ g) \circ \theta \colon \Omega \to \R$ is smooth;
by \Cref{thetabar on subset}, the function $h \circ g \colon \calO \to \R$
is smooth.
By \Cref{mfld}, $g$ is smooth.

If $f$ is the identity map on $\Omega$, then $g$ is the identity map
on $\calO$.
If maps $\Omega \xrightarrow{f} \Omega' \xrightarrow{f'} \Omega''$
descend, respectively, to maps 
$\calO \xrightarrow{g} \calO' \xrightarrow{g'} \calO''$, 
then the composition $f' \circ f$ descends to the composition $g' \circ g$.
These two facts imply that if $f$ is a diffeomorphism 
then $g$ is a diffeomorphism.
\end{proof}

%-------------------------------------------
\section{The invariants}
%-------------------------------------------
\labell{sec:invariants}

In this section, we prove the lemmas from \Cref{sec:defs} 
that we used to state our main theorem, \Cref{thm:main}.

Recall that we have fixed a $d$ dimensional torus $T$.

\begin{lemma} \labell{structure on chart}
Let $M$ be a locally standard $T$-manifold, and let $Q := M/T$,
equipped with the quotient differential structure (\Cref{quotient}).
Fix a locally standard $T$-chart 
$$\phi \colon U_M \to \Omega,$$ 
with $U_M$ a $T$-invariant open subset of $M$
and $\Omega$ a $\T^d$-invariant open subset 
of $\C^n \times \T^l \times \R^m$, with $n+l=d$.
Let $U_Q := \pi(U_M)$ be the image of $U_M$ under the quotient map 
$\pi \colon M \to Q$, and let 
$\calO := \theta(\Omega) \subset \Rplus^n \times \R^m$
be the image of $\Omega$ under the model quotient map \eqref{theta}.
There exists a unique map $\psi \colon U_Q \to \calO$
with which the diagram 
\begin{equation} \labell{square phi theta}
\xymatrix{
 U_M \ar[r]^{\phi} \ar[d]^{\pi} & \Omega \ar[d]^{\theta} \\
 U_Q \ar[r]^{\psi} & **[r] \calO 
} 
\end{equation}
commutes. Moreover, the following holds.
\begin{enumerate}
\item \labell{psi is diffeo}
$U_Q$ and $\calO$ are open subsets of $Q$ and of $\Rplus^n \times \R^m$,
and the map $\psi$ is a diffeomorphism between them
(with respect to their subspace differential structures).
\item \labell{preimage stab}
There exists a basis $\eta_1,\ldots,\eta_d$ of the integral lattice $\tZ$ 
such that,
for each $x \in\mathring{\del}^1 Q \cap U_Q$ and $j \in\{1,\ldots,n\}$,
if the $j$th coordinate of $\psi(x)$ vanishes, 
then the preimage of $x$ in $M$ is a $T$-orbit whose stabilizer
is the circle subgroup of $T$ that is encoded by $\etahat_j$.
\item \labell{preimage free}
For each $x \in U_Q$,
its preimage in $M$ is a $T$ orbit with trivial stabilizer
if and only if the first $n$ coordinates of $\psi(x)$ are all non-zero.
\end{enumerate}
\end{lemma}

\begin{proof}
By the definition of the quotient topology, $U_Q$ is open in $Q$.
By \Cref{thetabar is homeo}, $\calO$ is open in $\Rplus^n \times \R^m$.
Because $\phi$ is an equivariant diffeomorphism
with respect to an isomorphism $T \to \T^d$ and $\theta$ is $\T^d$ invariant,
we have a commuting diagram
\begin{equation} \labell{square phi theta rev}
\xymatrix{
 U_M \ar[r]^{\phi} \ar[d]^{\pi} & \Omega \ar[d] \ar[rd]^{\theta} & \\
 U_Q \ar[r]^{\ol\phi} & **[r] \Omega/\T^d \ar[r]^{\quad \thetabar} & \calO \,,
} 
\end{equation}
and $\ol\phi$ is a diffeomorphism 
from $U_Q$, with its quotient differential structure induced from~$U_M$,
to $\Omega/\T^d$, with its quotient differential structure induced 
from $\Omega$.
\Cref{thetabar on subset} implies that $\thetabar$ is a diffeomorphism
of differential spaces.
By \Cref{subquotient}, the quotient differential structure on $U_Q$
induced from~$U_M$
coincides with the subspace differential structure on $U_Q$ induced from $Q$.
Composing, we conclude that $\psi := \ol\phi \circ \ol\theta$
is a diffeomorphism, proving \eqref{psi is diffeo}.

For each $y \in \calO$, the preimage $\theta^{-1}(y)$ 
is a $\T^d$ orbit in $\C^n \times \T^l \times \R^m$.
Its stabilizer in $\T^d$ is the subtorus of $\T^d$ 
whose Lie algebra is spanned by the $j$th standard basis elements 
for those $j \in \{1,\ldots,n\}$ for which $y_j=0$.

Let $\alpha_1,\ldots,\alpha_d$ be the basis of $\tZ^*$
such that $\phi$ is equivariant with respect to the isomorphism
$\rho_{\alpha_1,\ldots,\alpha_d} \colon T \to \T^d$,
and let $\eta_1,\ldots,\eta_d$ be the dual basis of $\tZ$.
Let $x \in U_Q$. Let $y = \psi(x)$.
Then $\pi^{-1}(x)$ is a $T$ orbit in $M$,
whose stabilizer in $T$ is the image under the inverse isomorphism
$\rho_{\alpha_1,\ldots,\alpha_d}^{-1} \colon \T^d \to T$
of the stabilizer of $\theta^{-1}(y)$ in $\T^d$.
By the preceding paragraph and \Cref{stab etaj},
the stabilizer of $\pi^{-1}(x)$ in $T$ 
is the subtorus of $T$ whose Lie algebra is
spanned by the elements $\eta_j$ for those $j \in \{1,\ldots,n\}$
for which $\psi_j(x)=0$.
This implies \eqref{preimage stab} and~\eqref{preimage free}
\end{proof}

\begin{proof}[Proof of Lemma~\ref{l:quotient}]
Let $M$ be a locally standard $T$-manifold, and let $Q = M/T$.
We need to prove that there exists a unique manifold-with-corners structure
on $Q$ whose differential structure coincides with the quotient differential
structure on $Q$ that is induced from $M$.

By Part~\eqref{psi is diffeo} of \Cref{structure on chart},
for every locally standard $T$-chart on $M$ 
with domain $U_M$, the image of $U_M$ in $Q$ is an open subset $U_Q$ of $Q$
that is diffeomorphic to an open subset of $\Rplus^n \times \R^m$.
By \Cref{locally standard criterion}, every point in $Q$ 
is contained in such a $U_Q$. 
By \Cref{mfld}, this completes the proof of the lemma.
\end{proof}

\begin{proof}[Proof of Lemma~\ref{l:labelling}]
Let $M$ be a locally standard $T$-manifold, and let $Q = M/T$.
We need to prove that the preimage in $M$
of each point of $\mathring{\del}^1Q$ is an orbit
whose stabilizer is a circle subgroup of $T$,
and that the map that associates to each point in $\mathring{\del}^1Q$ 
the element of $\tZhat$ that encodes the stabilizer of its preimage
is a unimodular embedding.

By Parts~\eqref{psi is diffeo} and~\eqref{preimage stab}
of \Cref{structure on chart},
from any locally standard $T$-chart on $M$ with domain $U_M$,
we obtain a chart-with-corners
$$ \psi \colon U_Q \to \calO $$
on $Q$ whose domain $U_Q$ is the image of $U_M$
and whose image is open in $\Rplus^n \times \R^m$,
and a basis $\eta_1,\ldots,\eta_d$ of $\tZ$,
such that the following holds. For each $x \in\mathring{\del}^1 Q \cap U_Q$
and $j \in\{1,\ldots,n\}$,
if the $j$th coordinate of $\psi(x)$ vanishes, 
then the preimage of $x$ in $M$ is a $T$-orbit whose stabilizer
is the circle subgroup of $T$ that is encoded by $\etahat_j$.
By \Cref{locally standard criterion}, every point in $Q$ 
is contained in such a $U_Q$.
By \Cref{def:unimodular} of a unimodular labelling, 
this completes the proof of the lemma.
\end{proof}

\begin{proof}[Proof of Lemma~\ref{l:cM}]
Let $M$ be a manifold with a locally standard $T$-action.
Let $\Mfree$ be the set of points of $M$ whose $T$ stabilizer is trivial.
We need to prove that $\Mfree$ is the preimage in $M$
of the interior $\intQ$ of the manifold-with-corners $Q := M/T$,
that 
the map $\pi|_{\Mfree} \colon \Mfree \to \intQ$ is a principal $T$-bundle,
and that
there exists a unique cohomology class $c_M \in H^2(Q;\tZ)$
whose pullback to $H^2(\intQ;\tZ)$ is the Chern class
(\Cref{rk:Tbundles})
of this principal $T$-bundle.

By Parts~\eqref{psi is diffeo} and~\eqref{preimage free}
of \Cref{structure on chart},
from any locally standard $T$-chart on $M$ with domain $U_M$,
we obtain a chart-with-corners
$$ \psi \colon U_Q \to \calO $$
on $Q$ whose domain $U_Q$ is the image of $U_M$
and whose image is open in $\Rplus^n \times \R^m$,
such that the following holds.
For each $x \in U_Q$,
its preimage in $M$ is a $T$ orbit with trivial stabilizer
if and only if the first $n$ coordinates of $\psi(x)$ are all non-zero.
By \Cref{locally standard criterion}, every point in $Q$ 
is contained in such a $U_Q$.
It follows that $\Mfree$ is the preimage of $\intQ$.
We obtain local trivializations of $\Mfree \to \intQ$
from the diffeomorphisms
$\Omega_{\free} \to (\calO \cap (\Rpos^n \times \R^m)) \times \T^d$
that are given by $(z_1,\ldots,z_n;b_1,\ldots,b_l;x_1,\ldots,x_m) \mapsto 
((|z_1|^2\ldots,|z_n|^2;x_1,\ldots,x_m),
(\frac{z_1}{|z_1|},\ldots,\frac{z_n}{|z_n|}, b_1,\ldots,b_l)).$
The inclusion $\intQ \to Q$ is a homotopy equivalence; this implies that
the restriction $H^2(Q,\tZ) \to H^2(\intQ,\tZ)$ is an isomorphism.
%\ynote{justify why $\intQ \to Q$ is a homotopy equivalence?}
\end{proof}

%-------------------------------------------
\section{Classification modulo cutting: category-theoretical setup}
%-------------------------------------------
\labell{sec:proof modulo cutting}

Recall from \Cref{def:functor} and that we have a category $\frakM_\invt$
of locally standard $T$-manifolds $M$ and their equivariant diffeomorphisms,
a category $\frakQ_\invt$ of decorated manifolds-with-corners $Q$
and their decoration-preserving diffeomorphisms,
and a functor $\frakM_\invt \to \frakQ_\invt$,
taking $M$ to $Q := M/T$.
Recall from \Cref{rk:functor} that we would like to show 
that the functor $\frakM_\invt \to \frakQ_\invt$ 
is full and essentially surjective. 
We will show this with the help of a third category, $\frakP_\invt$,
of principal $T$-manifolds over decorated manifolds-with-corners
and their decoration-preserving equivariant diffeomorphisms.
In this section we sketch the argument:
we introduce the category $\frakP_\invt$, we assume that there exists a
\emph{cutting functor} $\frakP_\invt \to \frakM_\invt$ with certain properties,
and we conclude that the functor $\frakM_\invt \to \frakQ_\invt$
is full and essentially surjective, as required
(\Cref{what we need 1}, \Cref{what we need 2}).
The construction of the cutting functor and the proof of its
properties occupy the later sections.

It turns out that 
the cutting functor $\frakP_\invt \to \frakM_\invt$
is an equivalence of categories (\Cref{equiv of cat}).
This allows us to heuristically think of locally standard $T$ manifolds
as if they were principal $T$-bundles over manifolds with corners.

\begin{remark} \labell{rk:transverse maps}
It is natural to consider categories $\frakP$, $\frakM$, and $\frakQ$,
with the same objects as in $\frakP_\invt$, $\frakM_\invt$, and $\frakQ_\invt$
but with more morphisms.
There are various choices for such categories.
If we simply replace ``diffeomorphism'' by ``smooth'' everywhere,
we can still get a functor $\frakM \to \frakQ$,
but this functor is not full,
and the cutting construction does not produce a functor 
$\frakP \to \frakM$.
We do expect to be able to replace ``diffeomorphism'' 
by ``local diffeomorphism'' everywhere.
More generally, we expect to be able to require the morphisms in $\frakQ$ 
to be \emph{transverse maps} as in \cite{cutting},
with an analogous condition in $\frakM$. 
Details will appear elsewhere.
\eor
\end{remark}

We begin with some general category-theoretical vocabulary.
Keeping in mind \Cref{rk:transverse maps}, we work in a slightly more 
general setting than we strictly require, 
in that we allow categories in which not all arrows are isomorphisms.

The \emph{core} of a category $\frakA$ is the category $\frakA_\invt$
with the same objects and the invertible arrows.
Any functor $\frakA \to \frakB$ restricts to a functor
$\frakA_\invt \to \frakB_\invt$.
The functor $\frakA \to \frakB$ is essentially surjective
if and only if the functor $\frakA_\invt \to \frakB_\invt$
is essentially surjective.
We say that a functor $\frakA \to \frakB$ is \emph{full on isomorphisms}
if the functor $\frakA_\invt \to \frakB_\invt$ is full.\footnote{
If a functor is faithful, then ``full'' implies ``full on isomorphisms''.
In general, ``full on isomorphisms'' does not imply ``full'',
and ``full'' does not imply ``full on isomorphisms''.}

For any category $\frakA$, we can consider 
the set\footnote{Strictly speaking, it might be a class rather than a set.}
$C_\frakA$
of isomorphism classes of its objects.
Then $C_\frakA = C_{\frakA_\invt}$.
A functor $\frakA \to \frakB$ induces a map 
$C_\frakA \to C_\frakB$.
This map is onto if and only if the functor is essentially surjective. 
It is one-to-one if the functor is full on isomorphisms.

\begin{lemma}[Functor chasing] \labell{functor chasing}
Fix categories $\frakP$,\  $\frakM$, and $\frakQ$,
and functors $\frakP \to \frakM$,\  $\frakM \to \frakQ$, 
and $\frakP \to \frakQ$.
Assume that there exists a natural isomorphism 
from the given functor $\frakP \to \frakQ$ 
to the composed functor $\frakP \to \frakM \to \frakQ$.
\begin{enumerate}
\item[(1)]
If the functor $\frakP \to \frakQ$ is essentially surjective,
then the functor $\frakM \to \frakQ$ is essentially surjective.
\item[(2)]
If the functor $\frakP \to \frakQ$ is full
and the functor $\frakP \to \frakM$ is essentially surjective,
then the functor $\frakM \to \frakQ$ is full.
\item[(2')]
If the functor $\frakP \to \frakQ$ is full on isomorphisms
and the functor $\frakP \to \frakM$ is essentially surjective,
then the functor $\frakM \to \frakQ$ is full on isomorphisms.
\end{enumerate}
%\ynote{We refer to \Cref{functor chasing} 
%at the bottom of \Cref{sec:proof modulo cutting}.}
\end{lemma}

\begin{proof}
Because the given functor $\frakP \to \frakQ$
is naturally isomorphic to the composed functor $\frakP \to \frakM \to \frakQ$,
if one of them is essentially surjective, then so is the other.
Similarly, if one of them is full then so is the other,
and if one of them is full on isomorphisms then so is the other.
Moreover, the restricted functor $\frakP_\invt \to \frakQ_\invt$
is then naturally isomorphic to the composed functor 
$\frakP_\invt \to \frakM_\invt \to \frakQ_\invt$.
So we may assume that $\frakP \to \frakQ$ is equal 
to the composition $\frakP \to \frakM \to \frakQ$
and that $\frakP_\invt \to \frakQ_\invt$ is equal 
to the composition $\frakP_\invt \to \frakM_\invt \to \frakQ_\invt$.

We prove (1):

Let $Q$ be an object in $\frakQ$.
Because the functor $\frakP \to \frakQ$ is essentially surjective,
there exists an object $P$ in $\frakP$ 
whose image in $\frakQ$ is isomorphic to~$Q$.
The image of $P$ in $\frakM$ is then an object in $\frakM$
whose image in $\frakQ$ is isomorphic to $Q$.

We prove (2):

Let $M_1$ and $M_2$ be objects in $\frakM$,
let $Q_{M_1}$ and $Q_{M_2}$ be their images in $\frakQ$,
and let $f \colon Q_{M_1} \to Q_{M_2}$ be an arrow in~$\frakQ$.
Because the functor $\frakP \to \frakM$ is essentially surjective,
there exist objects $P_1$ and $P_2$ in $\frakP$,
with images $M_{P_1}$, $M_{P_2}$ in $\frakM$, and isomorphisms
$$ M_1 \cong M_{P_1} \quad \text{ and } \quad 
   M_2 \cong M_{P_2} .$$
Choose such objects and isomorphisms.
Applying the functor $\frakM \to \frakQ$, we get isomorphisms
$$ Q_{M_1} \cong Q_{P_1} \quad \text{ and } \quad 
   Q_{M_2} \cong Q_{P_2} ,$$
where $Q_{P_1}$ and $Q_{P_2}$ are the images of $P_1$ and $P_2$ in $\frakQ$.
Because the functor $\frakP \to \frakQ$ is full, the composition
$$ Q_{P_1} \cong Q_{M_1} \stackrel{f}{\to} Q_{M_2} \cong Q_{P_2} $$ 
has a lifting:
$$ P_1 \to P_2.$$
Applying the functor $\frakP \to \frakM$, 
we obtain an arrow 
$$ M_{P_1} \to M_{P_2} $$
that lifts the arrow $Q_{M_{P_1}} \to Q_{M_{P_2}}$.
The composition
$$ M_1 \cong M_{P_1} \to M_{P_2} \cong M_2 $$
then lifts the given arrow 
$$ Q_{M_1} \stackrel{f}{\to} Q_{M_1}.$$

We prove (2'):

Apply (2) to the restricted functors
$\frakP_\invt \to \frakM_\invt$ and $\frakM_\invt \to \frakQ_\invt$.
\end{proof}

%%%%%%%%%%%%%%%%%%%%%%%%%%%%%%%%%%%%%%%%%%%%%%%%%%%%%%%%%%%%5

A \textbf{principal $T$-bundle over a decorated manifold-with-corners}
is a triple $(P,Q,\lambdahat)$, where $Q$ is a manifold-with-corners, 
$\lambdahat \colon \mathring{\del^1}  Q \to \tZhat$ is a unimodular labelling,
and $\Pi \colon P \to Q$ is a principal $T$-bundle.
Sometimes we abbreviate and write $P$ instead of $(P,Q,\lambdahat)$.

\begin{definition} \labell{def:bundle quotient functor}
Let $\frakP_\invt$ denote the category whose objects are the 
principal $T$-bundles over decorated manifolds-with-corners $(P,Q,\lambdahat)$,
and whose morphisms are those equivariant diffeomorphisms $P_1 \to P_2$ 
whose induced diffeomorphism $Q_1 \to Q_2$ between the base spaces
respects their unimodular labellings.
Because $Q_1 \to Q_2$ intertwines the Chern classes of $P_1$ and $P_2$
(\Cref{rk:Tbundles}),
the map $(P,Q,\lambdahat) \mapsto (Q,\lambdahat,c)$,
where $c$ is the Chern class of $\Pi \colon P \to Q$, defines a functor 
from the category $\frakP_\invt$ to the category $\frakQ_\invt$
(\Cref{def:functor}).
We call it the \textbf{bundle quotient functor}.
\eod
\end{definition}

\begin{proposition} \labell{prop:bundle quotient functor}
The bundle quotient functor $\frakP_\invt \to \frakQ_\invt$ 
is essentially surjective
and is full (and full on isomorphisms).
\end{proposition}

\begin{proof}
This follows from the classification of principal $T$-bundles 
by their Chern classes (\Cref{rk:Tbundles}).
\end{proof}

We have now introduced the \emph{bundle quotient functor}
$\frakP_\invt \to \frakQ_\invt$.
Recall that we also have the \emph{quotient functor}
$\frakM_\invt \to \frakQ_\invt$ (\Cref{def:functor}).
In the remainder of this paper we will construct 
the \emph{cutting} functor $\frakP_\invt \to \frakM_\invt$
and the \emph{simultaneous toric radial blowup} functor 
$\frakM_\invt \to \frakP_\invt$,
which will have the following properties: 

\begin{proposition} \labell{natural isos}
\begin{enumerate}
\item \labell{natural iso P M Q}
The composition $\frakP_\invt \to \frakM_\invt \to \frakQ_\invt$
is naturally isomorphic
to the bundle quotient functor ${\frakP_\invt \to \frakQ_\invt}$.
\item \labell{natural iso M P M}
The composition $\frakM_\invt \to \frakP_\invt \to \frakM_\invt$
is naturally isomorphic to the identity functor on $\frakM_\invt$.
\item \labell{natural iso P M P}
The composition $\frakP_\invt \to \frakM_\invt \to \frakP_\invt$
is naturally isomorphic to the identity functor on $\frakP_\invt$.
\end{enumerate}
\end{proposition}

\begin{proof}[Pointers to the proof]
In Sections \ref{sec:cut-1}--\ref{sec:cutting} 
we construct the cutting functor $\frakP_\invt \to \frakM_\invt$
and a natural isomorphism
from the composition $\frakP_\invt \to \frakM_\invt \to \frakQ_\invt$
to the bundle quotient functor ${\frakP_\invt \to \frakQ_\invt}$.
In Sections \ref{sec:blowups}--\ref{sec:blowup smoothly}
we construct the simultaneous toric radial blowup functor
$\frakM_\invt \to \frakP_\invt$,
a natural isomorphism from the composition 
$\frakM_\invt \to \frakP_\invt \to \frakM_\invt$
to the identity functor on $\frakM_\invt$,
and a natural isomorphism from the composition 
$\frakP_\invt \to \frakM_\invt \to \frakP_\invt$
to the identity functor on $\frakP_\invt$.
\end{proof}

These results will allow us to prove the properties 
of the quotient functor $\frakM_\invt \to \frakQ_\invt$
that we need:

\begin{corollary} \labell{what we need 1}
(Assuming Part~\eqref{natural iso P M Q} of \Cref{natural isos},)
the quotient functor $\frakM_\invt \to \frakQ_\invt$
is essentially surjective.
\end{corollary}

\begin{proof}
By Part~\eqref{natural iso P M Q} of \Cref{natural isos}
and by Proposition~\ref{prop:bundle quotient functor},
the bundle quotient functor $\frakP_\invt \to \frakQ_\invt$ 
is naturally isomorphic to the composed functor
$\frakP_\invt \to \frakM_\invt \to \frakQ_\invt$
and is essentially surjective.
By Part~(1) of \Cref{functor chasing},
we conclude that the quotient functor $\frakM_\invt \to \frakQ_\invt$
is essentially surjective.
\end{proof}

\begin{corollary} \labell{what we need 2}
(Assuming Parts~\eqref{natural iso P M Q} and~\eqref{natural iso M P M} 
of \Cref{natural isos},)
the quotient functor $\frakM_\invt \to \frakQ_\invt$
is full (and full on isomorphisms).
\end{corollary}

\begin{proof}
By Part~\eqref{natural iso P M Q} of \Cref{natural isos}
and by Proposition~\ref{prop:bundle quotient functor},
the bundle quotient functor $\frakP_\invt \to \frakQ_\invt$ 
is naturally isomorphic to the composed functor
$\frakP_\invt \to \frakM_\invt \to \frakQ_\invt$
and is full (and full on isomorphisms).
By Part \eqref{natural iso M P M} of \Cref{natural isos},
the cutting functor $\frakP_\invt \to \frakM_\invt$ is essentially surjective.
By Part~(2) (or~(2')) of \Cref{functor chasing}, 
we conclude that the quotient functor $\frakM_\invt \to \frakQ_\invt$
is full (and full on isomorphisms).
\end{proof}

Additionally, by Parts~\eqref{natural iso M P M} 
and~\eqref{natural iso P M P} of \Cref{natural isos},

\begin{corollary} \labell{equiv of cat}
The cutting functor $\frakP_\invt \to \frakM_\invt$
is an equivalence of categories.
\end{corollary}

%-------------------------------------------
\section{The cutting functor, topologically}
%-------------------------------------------
\labell{sec:cut-1}

This section contains the first step in constructing the cutting functor
$\frakP_\invt \to \frakM_\invt$:
we construct the cut spaces as topological spaces,
yielding a functor from the category $\frakP_\invt$
to the category of topological $T$-spaces and their equivariant
homeomorphisms.

\begin{construction} \labell{Tx}
Let $\lambdahat \colon \mathring{\del}^1 Q \to \tZhat$
be a unimodular labelling on a manifold-with-corners.
We associate to each $x \in Q$ a sub-torus $T_{(x)} \subset T$, as follows.
Fix a unimodular chart $(\psi,(\eta_1,\ldots,\eta_d))$ (\Cref{def:unimodular})
whose domain contains $x$,
with $\psi \colon U_Q \to \calO \subset \Rplus^n \times \R^m$.
Then $T_{(x)}$ is the torus whose Lie algebra
is spanned by 
$\{ \eta_i \ | \ i \in \{1,\ldots,n\} \text{ and } \psi_i(x) = 0 \}$.
\eoc
\end{construction}

\begin{remark}
In the setup of \Cref{Tx}, the set
$\{ \pm \eta_i \ | \ i \in \{1,\ldots,n\} \text{ and } \psi_i(x) = 0 \}$
is the intersection of the images
$\lambdahat(U \cap \mathring{\del}^1 Q)$
where $U$ runs over all neighbourhoods of $x$.
Thus, the subtorus $T_{(x)}$ is well-defined:
it depends only on $\lambdahat$;
it is independent of the choice of unimodular chart.
\eor
\end{remark}

\begin{construction} \labell{topological cut space}
Let $(P,Q,\lambdahat)$ be a principal $T$-bundle over a decorated
manifold-with-corners.
Define an equivalence relation $\sim$ on $P$, as follows.
For each $p$ and $p'$ in $P$, we say that $p \sim p'$ if 
$\Pi(p) = \Pi(p') =: x$ and $p$ and $p'$ are in the same $T_{(x)}$ orbit,
where $T_{(x)} \subset T$ is the subtorus of \Cref{Tx}.
The \textbf{topological cut space} is the quotient, 
$M := \Pcut := P/{\sim}$, equipped with the quotient topology.
The \textbf{cutting map} is the quotient map $c \colon P \to M$.
\eoc
\end{construction}

\begin{remark} \labell{c not open map}
The cutting map is not an open map.
\eor
\end{remark}

\begin{remark} \labell{cutting on maps}
The $T$-action on $P$ descends to a continuous $T$-action 
on the cut space~$M$.
Moreover, every morphism $P \to P'$ in $\frakP$
descends to an equivariant homeomorphism $M \to M'$
between the corresponding cut spaces.
Thus, the maps $P \mapsto M$ and $(P \to P') \mapsto (M \to M')$
define a functor from the category $\frakP_\invt$
to the category of topological $T$-spaces and their equivariant homeomorphisms.
\eor
\end{remark}

\begin{remark} \labell{UM is UP cut}
Let $\Pi \colon P \to (Q,\lambdahat)$ be a principal $T$-bundle
over a decorated manifold-with-corners, let $M:=\Pcut$,
let $U_Q$ be an open subset of $Q$, let $U_P$ be its preimage in $P$,
and let $U_M$ be its preimage in $M$.
Then $U_M$ coincides with $(U_P)_{\cut}$ as sets, 
and, moreover, as topological $T$ spaces.
%\ynote{We refer to \Cref{UM is UP cut} in the proof of \Cref{cut smooth} 
%and \Cref{natural}.}
\eor
\end{remark}

\noindent

In Section \ref{sec:cutting} we describe a smooth manifold structure
on the topological cut space $M$ with which this functor becomes a functor
from the category $\frakP_\invt$ to the category $\frakM_\invt$. 
In order to do this, we must first understand how to pass
from diffeomorphisms of local models for $P$ 
to diffeomorphisms of local models for $M$; 
we do this in Section~\ref{sec:P to M}.

%-------------------------------------------
\section{Diffeomorphisms of models descend: from $P$ to $M$}
%-------------------------------------------
\labell{sec:P to M}

In this section we work purely within models.
The main result of this section --- \Cref{descend diffeo from P to M} ---
provides the central technical ingredient that we need
for our construction of the cutting functor in Section \ref{sec:cutting}.
\Cref{descend diffeo from P to M}
is also relevant for Wiemeler's work; see \Cref{rk:wiemeler}.

For any non-negative integers $n,l,m$, we 
consider the model quotient map $\theta$ \eqref{theta} and 
the \textbf{model cutting map} $c^\std$:
\begin{equation} \labell{model maps}
 \xymatrix{
\Rplus^n \times \R^m \times \T^{n+l}
 \ar[rr]^{c^{\std}}
   && \C^n \times \T^l \times \R^m
 \ar[rr]^{\theta}
   && \Rplus^n \times \R^m , 
}
\end{equation}
\begin{multline*}
 (s_1,\ldots,s_n;x_1,\ldots,x_m;b_1,\ldots,b_{n+l})
 \mapsto (z_1,\ldots,z_n;b_{n+1},\ldots,b_{n+l};x_1,\ldots,x_m)  \\
 \mapsto (s_1,\ldots,s_n;x_1,\ldots,x_m) , 
\end{multline*}
given, respectively, by $z_j = \sqrt{s_j}b_j$ and $s_j = |z_j|^2$.
Their composition is the projection map,
$$ \Rplus^n \times \R^m \times \T^{n+l} 
   \xrightarrow{\ \qquad \Pi \qquad \ } \Rplus^n \times \R^m.$$

\begin{proposition} \labell{product to Omega}
Fix non-negative integers $n$, $m$, and $l$.
Let $\calO$ be an open subset of $\Rplus^n \times \R^m$,
and let $\Omega:= \theta^{-1}(\calO) \subset \C^n \times \T^l \times \R^m$.
Let $\sim$ be the equivalence relation on $\calO \times \T^{n+l}$
where $p \sim p'$ iff $\Pi(p)=\Pi(p') =: (s_1,\ldots,s_n;x_1,\ldots,x_m)$
and $p,p'$ differ by an element of the subtorus of $\T^{n+l}$
whose Lie algebra is spanned by the $j$th standard basis elements 
for those $j \in \{1,\ldots,n\}$ for which $s_j=0$.

Then there exists a unique $\T^{n+l}$-action 
on $(\calO \times \T^{n+l})/{\sim}$
such that the quotient map 
$\calO \times \T^{n+l} \to (\calO \times \T^{n+l})/{\sim}$ 
is $\T^{n+l}$-equivariant, and there exists a unique map $\ol{c}^\std$
such that the following diagram commutes.
$$
\xymatrix{
\quad
\calO \times \T^{n+l} \ar[d] \ar@/^/[rrd]^{c^\std} && \\
\ \qquad 
(\calO \times \T^{n+l})/{\sim} \ar[rr]^(.6){\ol{c}^\std} && \Omega .
}
$$
Moreover, $\ol{c}^\std$ is a $\T^{n+l}$-equivariant homeomorphism.
%\ynote{\\ We refer to \Cref{product to Omega} 
%from the poof of \Cref{cut smooth}.}
\end{proposition}

\begin{proof}
Existence of the $\T^{n+l}$-action follows from $\T^{n+l}$ being abelian 
and $\Pi$ being $\T^{n+l}$ invariant.
Uniqueness of the $\T^{n+l}$-action follows from the quotient map being onto.
By the formula for $c^{\std}$, 
the equivalence classes coincide with the level sets of $c^{\std}$.
This implies that there exists a unique map $\ol{c}^\std$
such that the diagram commutes,
and that the map $\ol{c}^\std$ is one-to-one.  
Because $c^{\std}$ is onto, $\ol{c}^{\std}$ is onto,
so $\ol{c}^\std$ is a bijection.
Because $c^\std$ is continuous and proper, so is $\ol{c}^\std$.
Being a continuous and proper bijection, $\ol{c}^\std$ is a homeomorphism.
Finally, because $c^{\std}$ is equivariant, 
so is $\ol{c}^\std$.
\end{proof}

\begin{proposition} \labell{descend diffeo from P to M}
Fix non-negative integers $n,l,m,n',l',m'$ such that
$n+l=n'+l'=:d$ and $n+m = n'+m'$
and a non-negative integer $k$ such that $k \leq \min\{n,n'\}$.
Let
$$ \rho \colon \T^d \to \T^d $$
be an isomorphism of Lie groups 
that fixes $\T^k \times \{1\}^{d-k}$.
Fix open subsets
$$ \calO \subset \Rplus^n \times \R^m
\quad \text{ and } \quad
   \calO' \subset \Rplus^{n'} \times \R^{m'} $$
that are contained, respectively, in $\Rplus^k \times \Rpos^{n-k} \times \R^m$
and in $\Rplus^k \times \Rpos^{n'-k} \times \R^{m'}$.
Fix a $\rho$-equivariant diffeomorphism
$$ G \colon \calO \times \T^d \to \calO' \times \T^d$$
that, for each $1 \leq j \leq k$,
takes the intersection of $\calO \times \T^d$
with the $j$th facet $\{ s_j=0 \}$ of its ambient space
to the intersection of $\calO' \times \T^d$
with the $j$th facet $\{ s_j=0 \}$ of its ambient space.
Let 
$$ \Omega \ \subset \ \C^n \times \T^l \times \R^m 
\qquad \text{ and } \qquad
   \Omega' \ \subset \ \C^{n'} \times \T^{l'} \times \R^{m'} $$
be the preimages of $\calO$ and of $\calO'$
under the respective model quotient maps~$\theta$.
Then there exist a unique map $f \colon \Omega \to \Omega'$
and a unique map $g \colon \calO \to \calO'$
such that the following diagram commutes:
$$ \xymatrix{
   \calO \times \T^d
       \ar[d]^{c^{\std}} \ar[rr]^{G} 
   && \calO' \times \T^d
              \ar[d]^{{c}^{\std}}  \\
%\ar@{.>}[r]^{f} 
\Omega \ar[d]^{\theta} \ar[rr]^{f} && \Omega' \ar[d]^{\theta} \\
 \calO \ar[rr]^{g} && \calO' .
} $$
Moreover, $f$ is a $\rho$-equivariant diffeomorphism,
and $g$ is a diffeomorphism.
%\ynote{ \\ Moreover:
%Assume only that $G$ is a $\rho$-equivariant smooth map
%that takes the intersection of $\calO \times \T^d$ with $\{ s_j=0 \}$
%to the intersection of $\calO' \times \T^d$ with $\{ s_j=0 \}$.
%There still exist a unique $f$ and a unique $g$
%such that the diagram commutes.
%Moreover, if $G$ is a diffeomorphism then so are $f$ and $g$,
%if $G$ is a local diffeomorphisms then so are $f$ and $g$.}
\end{proposition}

\begin{proof}
We write coordinates on $\T^d$ as
$$ \tau = (a,b) \quad \text{ with } \ \
   a = (a_1,\ldots,a_k), \ \text{ and } \ 
   b=(b_{k+1},\ldots,b_{d}). $$
Because the isomorphism $ \rho \colon \T^d \to \T^d $
fixes $\T^k \times \{1\}^{d-k}$, we have
$$ \rho (a,b) = \left( 
a_1 \rho_1(b),\, \ldots,\, a_k \rho_k(b),\, 
    \rho_{k+1}(b),\, \ldots,\, \rho_d(b) \, \right)$$
for some homomorphisms $\rho_j \colon \T^{d-k} \to S^1$. 

Because $G$ is $\rho$-equivariant, we can write
$$ G(s,x;\tau) \ = \ \,(\, g(s,x) \, ; \, \rho(\tau) \cdot A(s,x) \, ) 
 \quad \text{ for } (s,x) \in \calO \text{ and } \tau \in \T^d,$$
for some maps $g \colon \calO \to \calO'$ and $A \colon \calO \to \T^d$.
Because $G$ is smooth, $g$ and $A$ are smooth;
because $G$ is a diffeomorphism, $g$ is a diffeomorphism.

We write coordinates on $\calO$ as
$$ (s,x) \quad \text{ with } \ 
 s = (s_1,\ldots,s_n) \quad \text{ and } \ x = (x_1,\ldots,x_m).$$
By assumption, $s_1,\ldots,s_k$ are non-negative, and 
$s_{k+1},\ldots,s_{n}$ are strictly positive.
We write 
$$ g(s,x) = (s'_1(s,x),\ldots,s'_{n'}(s,x);x'_1(s,x),\ldots,x'_{m'}(s,x)),$$
where $s'_1,\ldots,s'_{n'}, x'_1,\ldots,x'_{m'}$ are smooth real-valued
functions on $\calO$,
with $s'_1,\ldots,s'_k$ non-negative
and $s'_{k+1},\ldots,s'_{n'}$ strictly positive.
Because 
for each $1 \leq j \leq k$ the diffeomorphism $g \colon \calO \to \calO'$ 
takes the facet $\{ s_j = 0 \}$ of $\calO$
to the facet $\{ s'_j = 0 \}$ of $\calO'$, 
by Hadamard's lemma\footnote{\labell{fn:hadamard} $s_j'(s,x)
 = \int_0^1 (\frac{d}{dt}s_j'(ts,x)) dt
 = \int_0^1 s \del_1(s_j')(ts,x) dt = s h_j(s,x)$
where $h_j = \int_0^1 \del_1(s_j')(ts,x) dt$.
The first equality is by the fundamental theorem of calculus;
the second is by the chain rule.
}
we can write
\begin{equation} \labell{hadamard}
 s'_j(s,x) = s_j h_j(s,x) \quad \text{ for all } 1 \leq j \leq k ,
\end{equation}
for smooth functions $h_j \colon \calO \to \R$.
Outside the $j$th facet $\{s_j=0\}$, we have $s_j > 0$ and $s'_j(s,x) > 0$,
so $h_j(s,x)$ is positive. 
By continuity, $h_j$ is non-negative along the $j$th facet.
But along the $j$th facet the partial derivatives of $s'_j(s,x)$
with respect to the coordinates other than $s_j$ vanish,
so the partial derivative of $s'_j(s,x)$ with respect to the $s_j$ variable
cannot vanish (because $g$ is a diffeomorphism). 
This partial derivative is $h_j$, so $h_j$ is non-vanishing 
along the $j$th facet.  We conclude that $h_j$ is positive everywhere.

We write the smooth map $A \colon \calO \to \T^d$ as 
$ A(s,x) = (A_1(s,x),\ldots,A_d(s,x)) $.  
The map $G \colon \calO \times \T^d \to \calO' \times \T^d$ 
then becomes
\begin{multline*}
 G(s,x;a,b) \ =  \ \\ \,(\, 
s_1 h_1(s,x), \, \ldots, \, s_k h_k(s,x) , \ 
s'_{k+1}(s,x), \, \ldots, \,s'_{n'}(s,x) \ ; 
x'_{1}(s,x), \, \ldots, \,x'_{m'}(s,x) \ ; \\
a_1\rho_1(b) A_1(s,x), \ \ldots, \ a_k\rho_k(b) A_k(s,x) \ ; 
\rho_{k+1}(b) A_{k+1}(s,x), \ \ldots, \ \rho_{d}(b) A_{d}(s,x) \ 
\, ) \,.
\end{multline*}
The model cutting map 
$c^\std \colon \Rplus^{n'} \times \R^{m'} \times \T^d
 \to \C^{n'} \times \T^{l'} \times \R^{m'}$ is
\begin{multline*}
 {c}^{\text{std}} (s'_1,\ldots,s'_{n'};x'_1,\ldots,x'_{m'};
a'_1,\ldots,a'_{k}, b'_{k+1},\ldots,b'_{d}) = \\
(\sqrt{s'_1}a'_{1},\ldots,\sqrt{s'_k}a'_k,
 \sqrt{s'_{k+1}}b'_{k+1},\ldots,\sqrt{s'_{n'}}b'_{n'};\,
   b'_{n'+1},\ldots,b'_{d};x'_{1'},\ldots,x'_{m'}).
\end{multline*}
Composing $G$ with this model cutting map, we obtain
\begin{multline} \labell{composing}
{c}^\text{std} (G (s,x;a,b)) = \\
\Big(
 \sqrt{s_1 h_1(s,x)} a_1 \rho_1(b) A_1(s,x), \ldots,
 \sqrt{s_k h_k(s,x)} a_k \rho_k(b) A_k(s,x), \\
 \sqrt{s'_{k+1}(s,x)} \rho_{k+1}(b) A_{k+1}(s,x), \ldots,
 \sqrt{s'_{n'}(s,x)} \rho_{n'}(b) A_{n'}(s,x); \\
 \rho_{n'+1}(b) A_{n'+1}(s,x), \ldots, \rho_{d}(b) A_{d}(s,x),  \\
 x'_1(s,x),\ldots,x'_{m'}(s,x) \Big) .
\end{multline}
In order for this composition to be equal to $f(c^{\text{std}}(s,x;a,b))$,
where here the model cutting map is 
$c^\std \colon \Rplus^{n} \times \R^{m} \times \T^d
 \to \C^{n} \times \T^{l} \times \R^{m}$, 
it is necessary and sufficient that $f \colon \Omega \to \Omega'$ 
will satisfy
\begin{multline*}
f(\sqrt{s_1}a_1,\ldots,\sqrt{s_k}a_k,
  \sqrt{s_{k+1}}b_{k+1},\ldots,\sqrt{s_n}b_n;
 b_{n+1},\ldots,b_{d};x_1,\ldots,x_m) \\
 = \text{ the right hand side of \eqref{composing} }.
\end{multline*}
Because for each $1 \leq j \leq k$
the right hand side of \eqref{composing} is independent of $a_j$ when $s_j=0$,
and because $s_{k+1},\ldots,s_{n}$ never vanish on $\calO$,
there exists a unique such $f \colon \Omega \to \Omega'$,
which we will now describe.

We write coordinates on $\Omega$ as
$$ (z,c,x) \quad \text{ where } \quad
 z = (z_1,\ldots,z_n) , \quad 
 c = (b_{n+1},\ldots,b_d) , \quad 
 x = (x_1,\ldots,x_m) $$
and on $\Omega'$ as
$$ (z',c',x') \quad \text{ where } \quad
 z' = (z'_1,\ldots,z'_{n'}) , \quad 
 c' = (b'_{n'+1},\ldots,b'_{d}) , \quad 
 x' = (x'_1,\ldots,x'_{m'}) .$$
By assumption, $z_{k+1},\ldots,z_n$ and $z'_{k+1},\ldots,z'_{n'}$ never vanish.
\begin{itemize}
\item
The last $m'$ components of $f \colon \Omega \to \Omega'$ are
$$x'_j(s,x), \quad \text{ for $1 \leq j \leq m'$}, $$
where $s = (|z_1|^2,\ldots,|z_n|^2)$,
so they are smooth as functions of $(z,c,x) \in \Omega$.
\end{itemize}
\begin{itemize}
\item
The middle $l'$ components of $f \colon \Omega \to \Omega'$ are
$$ \rho_{j}(b) A_{j}(s,x) 
   \quad \text{ for }  n'+1 \leq j \leq d,$$
with $s$ as before
and $b = (\dfrac{z_{k+1}}{|z_{k+1}|},\ldots,\dfrac{z_{n}}{|z_{n}|},
b_{n+1},\ldots,b_{d})$;
these are well defined and smooth as functions of $(z,c,x) \in \Omega$ 
because $z_{k+1},\ldots,z_n$ never vanish on $\Omega$
and $b_{n+1},\ldots,b_d$ are the components of $c$.
\item
The first $k$ components of $f$ are
$$
\sqrt{s_j h_j(s,x)} a_j \rho_j(b) A_j(s,x) 
 = z_j \sqrt{h_j(s,x)} \rho_j(b) A_j(s,x) 
\quad \text{ for } 1 \leq j \leq k,
$$
with $s$ and $b$ as before;
these are smooth as functions of $(z,c,x) \in \Omega$
because $h_1,\ldots,h_k$ are strictly positive.
\item
The remaining $n'-k$ components of $f$ are
$$ \sqrt{s'_j(s,x)} \rho_j(b) A_j(s,x)
\quad \text{ for } k+1 \leq j \leq n'$$
with $s$ and $b$ as before;
these are smooth as functions of $(z,c,x) \in \Omega$
because $s'_{k+1},\ldots,s'_n$ are strictly positive.
\end{itemize}
\end{proof}

%-------------------------------------------
\section{The cutting functor, smoothly}
%-------------------------------------------
\labell{sec:cutting}

In this section we complete the construction of the cutting functor
$\frakP_\invt \to \frakM_\invt$,
from the category of principal $T$-bundles
over decorated manifolds-with-corners,
to the category of locally standard $T$-manifolds,
and we will construct a natural isomorphism from its composition 
$\frakP_\invt \to \frakM_\invt \to \frakQ_\invt$
with the quotient functor
$\frakM_\invt \to \frakQ_\invt$
to the bundle quotient functor $\frakP_\invt \to \frakQ_\invt$.

In Section~\ref{sec:cut-1} we described the cutting functor topologically,
as a functor from $\frakP_\invt$
to the category of topological $T$-spaces and their equivariant homeomorphisms.
It remains to assign, for every object $P$ of $\frakP_\invt$,
a (smooth) manifold structure on the corresponding topological cut space $M$,
such that the continuous $T$-action on $M$ is smooth, 
and such that for every morphism $P \to P'$ in $\frakP_\invt$,
the induced equivariant homeomorphism $M \to M'$ is a diffeomorphism.
We will do this using charts.
(Warning: the quotient map $c \colon P \to M$ will not be smooth!
It will be modelled on the map 
that takes $(e^{i\theta},s) \in S^1 \times \Rplus$
to $\sqrt{s}e^{i\theta} \in \C$.)

We introduce a topological analogue of \Cref{def:Tchart}:

\begin{definition} \labell{def:Tchart v2}
Let $M$ be a topological $T$-space.
A \textbf{topological locally standard $T$-chart} on $M$
is a $\rho_{\alpha_1,\ldots,\alpha_d}$-equivariant homeomorphism 
$$ U_M \xrightarrow{\phi} \Omega $$
from a $T$-invariant open subset $U_M$ of $M$
to a $\T^d$-invariant open subset $\Omega$ of $\C^n \times \T^l \times \R^m$,
for some non-negative integers $n$ and $l$ such that $n+l=d$
and for some basis $\alpha_1,\ldots,\alpha_d$ of the weight lattice $\tZ^*$.
Two such charts
$$\phi' \colon U_M' \to \Omega'
\quad \text{ and } \quad
  \phi'' \colon U_M'' \to \Omega'',$$
are \textbf{smoothly compatible} if the composition 
$\phi'' \circ {\phi'}^{-1}$,
is a \emph{diffeomorphism} (and not only a homeomorphism)
from the open subset $\phi'( U_M' \cap U_M'' )$ of $\Omega'$  
to the open subset $\phi''( U_M' \cap U_M'' )$ of $\Omega''$.
\end{definition}

{
\begin{proposition} \labell{cut smooth}
Let $(P,Q,\lambdahat)$ be a principal $T$-bundle
over a decorated manifold-with-corners, 
and let $M=\Pcut$ be the corresponding (topological) cut space.
Fix an open subset $U_Q$ of $Q$, 
its preimages $U_P$ and $U_M$ in $P$ and in $M$, a unimodular chart
$$ (\psi,(\eta_1,\ldots,\eta_{n+l})),
\quad \psi \colon U_Q \to \calO \ \subset \Rplus^n \times \R^m $$
with $n+l=d =: \dim T$, and a trivialization 
\begin{equation} \labell{trivialization}
 U_P \xrightarrow{\cong} U_Q \times T
\end{equation}
of the principal bundle $P$ over $U_Q$.
Let $\alpha_1,\ldots,\alpha_{d}$ be the basis of $\tZ^*$ 
that is dual to the basis $\eta_1,\ldots,\eta_{d}$ of $\tZ$.
Let 
$ \tpsi \colon U_P \to \calO \times \T^d $
be the composition of the trivialization \eqref{trivialization} with the map 
$ \psi \times \rho_{\alpha_1,\ldots,\alpha_d} 
 \colon U_Q \times T \to \calO \times \T^d $.
Then there exists a unique map $\phi \colon U_M \to \Omega$ with which 
the following diagram commutes
$$ \xymatrix{
% P \ \supset & U_P \ar[r]^(.35){\tpsi} \ar[d]^{c} 
 P \ \supset & U_P \ar[r]^{\tpsi} \ar[d]_{c} 
   & **[r] {\calO \times \T^{n+l}} 
      \ar[d]^{c^{\std}} 
   & \subset \ \Rplus^n \times \R^m 
     \times \T^{n+l} \\
 M \ \supset & U_M \ar[d]_{\pi} \ar[r]^{\phi} & \Omega \ar[d]^{\theta} & 
 \subset \C^n \times \T^l \times \R^m  \\
 Q \ \supset & U_Q \ar[r]^{\psi} & \calO & \subset \ \Rplus^n \times \R^m\,,
} $$
and this map~$\phi$ is a topological locally standard $T$-chart.
Moreover, every two topological locally standard $T$-charts on $M$
that are obtained in this way are smoothly compatible. 
\end{proposition}

\begin{proof}
\Cref{UM is UP cut} and \Cref{product to Omega}
imply that $\phi$ exists, is unique,
is $\rho_{\alpha_1,\ldots,\alpha_d}$-equivariant, and is a homeomorphism.

Let $\phi$ and $\phi'$ be two locally standard $T$-charts
that are obtained in this way.
Restricting to the intersection of their domains, we may assume
that they have the same domain, $U_M$.
So we have diagrams 
$$ \xymatrix{
 U_P \ar[r]^{\tpsi} \ar[d]^{c} 
   & **[r] {\calO \times \T^d } 
      \ar[d]^{c^{\std}} \\ 
 U_M \ar[d]^{\pi} \ar[r]^{\phi} & \Omega \ar[d]^{\theta}  \\
 U_Q \ar[r]^{\psi} & \calO 
} \qquad \text{ and } \qquad
   \xymatrix{
 U_P \ar[r]^{\tpsi'} \ar[d]^{c} 
   & **[r] {\calO' \times \T^d } 
      \ar[d]^{c^{\std}} \\ 
 U_M \ar[d]^{\pi} \ar[r]^{\phi'} & \Omega' \ar[d]^{\theta}  \\
 U_Q \ar[r]^{\psi'} & \calO'
} $$
where $\tpsi$ is a $\rho$-equivariant diffeomorphisms
and $\tpsi'$ is a $\rho'$-equivariant diffeomorphisms
for isomorphisms 
$$\rho = \rho_{\alpha_1,\ldots,\alpha_d} \colon T \to \T^d 
\quad \text{ and } \quad
  \rho' = \rho_{\alpha'_1,\ldots,\alpha'_d} \colon T \to \T^d,$$
and where
$\phi$ is a $\rho$-equivariant homeomorphisms
and $\phi'$ is a $\rho'$-equivariant homeomorphisms.
Combining these, we obtain the following diagram,
in which $G := \tpsi' \circ {\tpsi}^{-1}$
is a $T$-equivariant diffeomorphism,
$g:= \psi' \circ {\psi}^{-1}$ is a diffeomorphism,
and $f:= \phi' \circ {\phi}^{-1}$ is a $T$-equivariant homeomorphism.
$$ \xymatrix{
   **[l] \calO \times \T^d 
       \ar[d]^{c^{\std}} \ar[r]^{G} 
   & **[r] \calO' \times \T^d 
              \ar[d]^{{c'}^{\std}} 
\\
\Omega \ar[d]^{\theta} \ar[r]^{f} & \Omega' \ar[d]^{\theta'} 
\\
 \calO \ar[r]^{g} & \calO' 
} $$
Here,
$$ \calO \subset \Rplus^n \times \R^m
\quad \text{ and } \quad
   \calO' \subset \Rplus^{n'} \times \R^{m'}, $$
and 
$$ \Omega \subset 
\C^n \times \T^{l} \times \R^m 
\quad \text{ and } \quad
   \Omega' \subset 
\C^{n'} \times \T^{l'} \times \R^{m'} .$$

After replacing the three models on the left column of our diagram
by models that differ from them by a permutation of the coordinates,
we may assume that 
the set $\calO$ meets the first $k$ facets
of the model $\Rplus^{n} \times \R^m$
and does not meet any of its other facets. 

After replacing the three models on the right column of our diagram
by models that differ from them by a permutation of the coordinates,
we may assume that
the set $\calO'$ also meets the first $k$ facets
of its model $\Rplus^{n'} \times \R^{m'}$
and does not meet any of its other facets,
and moreover that for each $1 \leq j \leq k$
the map $g$ takes the intersection of $\calO$ with the $j$th facet
of its model
to the intersection of $\calO'$ with the $j$th facet of its model.

After further replacing the three models
on the right column by models that differ from them
in the signs of some of the first $k$ weights, 
we may assume that
for all $1 \leq j \leq k$ we have $\eta_j = \eta_j'$,
where $\eta_1,\ldots,\eta_d$ is the dual basis to $\alpha_1,\ldots,\alpha_d$
and $\eta_1',\ldots,\eta_d'$ is the dual basis to $\alpha_1',\ldots,\alpha_d'$.
Then the isomorphism
$$ \rho_{\T^d} := 
\rho_{\alpha'_1,\ldots,\alpha'_d} \circ \rho_{\alpha_1,\ldots,\alpha_d}^{-1} 
\colon \T^d \to \T^d .$$
fixes $\T^k \times \{ 1 \}^{d-k}$, and the maps 
$$ G \colon \calO \times \T^d \to \calO' \times \T^d $$
and 
$$ f \colon \Omega \to \Omega' $$
are $\rho_{\T^d}$-equivariant.

By Proposition~\ref{descend diffeo from P to M},
the map $f \colon \Omega \to \Omega'$ is a diffeomorphism.
\end{proof}
}

\begin{corollary} \labell{cut smooth revised}
Let $P$ be a principal $T$-bundle over a decorated manifold-with-corners, 
and let $M:=\Pcut$ be the corresponding topological cut space
(\Cref{topological cut space}).
Then there exists a unique (smooth) manifold structure on $M$
with which the maps $\phi$ of \Cref{cut smooth} are diffeomorphisms.
With this structure, $M$ is a locally standard $T$-manifold.

Let $\tilde{\psi} \colon P \to P'$ be an equivariant diffeomorphism 
of principal $T$-bundles over decorated manifold-with-corners,
and let $\phi \colon M \to M'$ 
be the induced equivariant homeomorphism
on the cut spaces (\Cref{cutting on maps}).
Then $\phi$ is smooth. 
\end{corollary}

\begin{proof}
Fix $P$.
By \Cref{def:unimodular} and because $P \to Q$ is locally trivializable,
we can cover $Q$ by domains $U_Q$ of unimodular charts 
over which there exist local trivializations of $P$. 
By \Cref{cut smooth},
on the preimage in $M$ of each $U_Q$ in $M$
we obtain a locally standard $T$-chart, 
and these locally standard $T$-charts are pairwise smoothly compatible.  
By \Cref{mfld lemma1}, we obtain
a manifold structure on $M$ with the required properties.

Fix $\tilde{\psi} \colon P \to P'$,
and let $\psi \colon Q \to Q'$ be the induced map on their base spaces.
Fix a local trivialization $U_{P} \to U_{Q} \times T$ of $P$
and a unimodular chart $U_{Q} \to \calO \subset \Rplus^n \times \R^m$.
%\ynote{To allow local diffeos, choose $U_Q$ small enough 
%such that $\psi|_{U_Q}$ is a diffeo with its image.}
Through $\tilde{\psi}$ and $\psi$, these induce
a local trivialization $U_{P'} \to U_{Q'} \times T$ of $P'$
and a unimodular chart $U_{Q'} \to \calO \subset \Rplus^n \times \R^m$.
In the corresponding locally standard $T$-charts 
$U_M \to \Omega$ and $U_{M'} \to \Omega$,
the map $\phi|_{U_M} \colon U_M \to U_{M'}$ becomes the identity map
$\Omega \to \Omega$.
Because every point of $M$ is contained in such an open set $U_M$,
the map $M \to M'$ is smooth.
\end{proof}

%------------------------------------------------------------------

We have now defined the cutting functor $\frakP_\invt \to \frakM_\invt$.
Its composition $\frakP_\invt \to \frakM_\invt \to \frakQ_\invt$ 
with the quotient functor 
$\frakM_\invt \to \frakQ_\invt$ takes each $P$ to $\Pcut/T$.

\begin{lemma}
\labell{natural iso lemma}
Let $\Pi \colon P \to Q$ be a principal $T$-bundle over a decorated
manifold-with-corners, let $M := \Pcut$ 
be the corresponding locally standard $T$-manifold
and $c \colon P \to M$ the cut map, 
let $\pi \colon M \to M/T$ be the corresponding quotient map,
and equip $M/T$ with its decorated manifold-with-corners structure
that is induced from $M$ (\Cref{l:quotient,l:labelling}).
There exists a unique map $i_P \colon M/T \to Q$ 
such that the following diagram commutes.
$$
\xymatrix{
P \ar[drr]_{\Pi} \ar[r]^{c} & M \ar[r]^{\pi} & M/T \ar[d]^{i_P} \\
 & & Q 
}
$$
Moreover, the map $i_P$ 
is a diffeomorphism of manifolds-with-corners,
it pulls back the given unimodular labelling on $Q$
to the unimodular labelling on $M/T$ that is induced from $M$,
and it pulls back the Chern class of the principal $T$-bundle $P$
to the Chern class of the locally standard $T$-manifold $M$.
Moreover, for every isomorphism $P \to \hat{P}$ in $\frakP_\invt$,
the induced maps 
$\Pcut/T \to \hat{P}_{\cut}/T$ and $Q \to \hat{Q}$ fit into a commuting square
$$
\xymatrix{
\Pcut/T \ar[r] \ar[d]^{i_P} & \hat{P}_{\cut}/T \ar[d]^{i_{\hat{P}}} \\
Q \ar[r] & \hat{Q} \,.
}
$$
\end{lemma}

\begin{proof}
This follows from the construction.
%\ynote{add details?}
\end{proof}

\begin{proof}[Proof of Part~\eqref{natural iso P M Q} of \Cref{natural isos}]
By \Cref{natural iso lemma},
by assigning to every principal $T$-bundle $\Pi \colon P \to Q$ 
over a decorated manifold-with-corners $Q$ the isomorphism 
of decorated manifolds-with-corners 
$i_P \colon \Pcut/T \to Q$ that is induced from $\Pi$,
we obtain a natural isomorphism
from the bundle quotient functor $\frakP_\invt \to \frakQ_\invt$ 
to the composition $\frakP_\invt \to \frakM_\invt \to \frakQ_\invt$.
\end{proof}

Now that we have proved Part~\eqref{natural iso P M Q} of \Cref{natural isos},
we know by \Cref{what we need 1} that the quotient functor 
$\frakM_\invt \to \frakQ_\invt$ is essentially surjective.
For our classification of locally standard $T$-manifolds, 
by \Cref{rk:functor},
it remains to show that the quotient functor $\frakM_\invt \to \frakQ_\invt$ 
is also full (and full on isomorphisms).
By \Cref{what we need 2},
this will follow from Part~\eqref{natural iso M P M} of \Cref{natural isos}.
Thus, we are looking for a functor $\frakM_\invt \to \frakP_\invt$ 
such that the composition $\frakM_\invt \to \frakP_\invt \to \frakM_\invt$
is naturally isomorphic to the identity functor on $\frakM_\invt$.
We will do this in Sections~\ref{sec:blowups}--\ref{sec:blowup smoothly},
with the \emph{simultaneous toric radial blowup} functor
$\frakM_\invt \to \frakP_\invt$.
In fact, we will show more:
the simultaneous toric radial blowup functor will be an \emph{inverse}
(up to natural isomorphism) of the cutting functor,
so that Parts~\eqref{natural iso M P M} and~\eqref{natural iso P M P}
of \Cref{natural isos} will both hold. Thus, the cutting functor 
is an equivalence of categories (as stated in \Cref{equiv of cat}).

{

%=====================================
\section{Simultaneous toric radial blowups, set-theoretically}
%=====================================
\labell{sec:blowups}

In Sections~\ref{sec:cut-1}--\ref{sec:cutting},
we defined the cutting functor, $\frakP_\invt \to \frakM_\invt$,
from the category $\frakP_\invt$ of principal $T$-bundles
over decorated manifolds-with-corners
with their decoration-preserving equivariant diffeomorphisms,  
to the category $\frakM_\invt$ of locally standard $T$-manifolds
with their equivariant diffeomorphisms.
We would like to now show that this functor is essentially surjective.  
We will do this by constructing a functor 
$$\frakM_\invt \to \frakP_\invt,$$
which we call the \emph{simultaneous toric radial blowup} functor,
such that the compositions 
$$ \frakM_\invt \to \frakP_\invt \to \frakM_\invt
\qquad \text{ and } \qquad
 \frakP_\invt \to \frakM_\invt \to \frakP_\invt $$
are naturally isomorphic to the identity functors 
on $\frakM_\invt$ and on $\frakP_\invt$.

We begin, in this section, with constructing the 
simultaneous toric radial blowup
functor on the level of sets.  
To each locally standard $T$-manifold $M$,
we will associate a set $P_M$ with a $T$-action
and a $T$-equivariant map $c_M \colon P_M \to M$
(\Cref{c:simultaneous}). 
The $T$-action on $P_M$ will be free,
and the composition of $c_M$ with the quotient map $\pi \colon M \to Q := M/T$,
$$
\xymatrix{
  P_M \ar[r]^{c_M} %\ar@/_/[rr]_{\Pi} 
  & M\ar[r]^{\pi} & Q \,,
}
$$
will be a quotient map
(namely, it will induce a bijection from $P_M/T$ to $Q$)
(\Cref{free}).
Also,
to each $T$-equivariant diffeomorphism $\psi \colon M_1 \to M_2$
between locally standard $T$-manifolds with corresponding $T$-sets 
$P_{M_1}$ and $P_{M_2}$,
we will associate a $T$-equivariant bijection (\Cref{c:psi tilde})
$\tilde{\psi} \colon P_{M_1} \to P_{M_2}$,
such that the following diagram commutes.
$$
\xymatrix{
 P_{M_1} \ar[r]^{\tilde{\psi}} \ar[d]_{c_{M_1}} & P_{M_2} \ar[d]^{c_{M_2}} \\
 M_1 \ar[r]^{\psi} & M_2 \, .
}
$$

These will satisfy the following conditions (\Cref{r:psi tilde}).

\begin{itemize}
\item[(i)]
If $M_1 = M_2 =: M$ and $\psi \colon M \to M$ is the identity map, 
then $\tilde{\psi} \colon P_M \to P_M$ is also the identity map.

\item[(ii)]
For each composition $\psi_{31} := \psi_{32} \circ \psi_{21}$ 
of diffeomorphisms 
$M_1 \xrightarrow{\psi_{21}} M_2 \xrightarrow{\psi_{32}} M_3$
of locally standard $T$-manifolds, we will have 
$\tilde{\psi}_{31} = \tilde{\psi}_{32} \circ \tilde{\psi}_{21}
 \colon P_{M_1} \to P_{M_3}$.
\end{itemize}

%%-----------------------------------------------------------------------

We now begin this construction.

Let $M$ be a locally standard $T$-manifold,
with quotient $\pi \colon M \to Q := M/T$.

\begin{definition}
For each $\etahat = \pm \eta \in \tZhat = (\tZ)_{\primitive}/\pm 1 $,
the fixed point set $M^{S^1_\etahat}$ 
of the corresponding circle subgroup $S^1_\etahat$ of $T$
is a locally finite disjoint union of closed submanifolds of $M$.
We denote the union of its codimension-two components
by $M^\etahat$.
We call the sets $M^{\etahat}$ \textbf{toric divisors}
or \textbf{characteristic submanifolds}.
\end{definition}

\begin{remark} \labell{terms}
The term \emph{characteristic submanifolds}
aligns well with the analogous term in the toric topology community.
The term \emph{toric divisor}
aligns well with the analogous term in the algebraic geometry community.
\eor
\end{remark}

\begin{remark} \labell{indexing}
We do not restrict to individual connected components of the sets $M^\etahat$.
With this convention, the indexing set $\tZhat$ for the set of toric divisors
is the same for all locally standard $T$-manifolds $M$.
If $M$ is compact, the set $M^{\etahat}$ is empty for 
all but finitely many choices of $\etahat$.
\end{remark}

\begin{remark} \labell{Ix Tx}
Let $x \in M$, let $T_x$ be its stabilizer in $T$, and let 
\begin{equation*} 
I_x := \{ \etahat \in \tZhat \ | \ x \in M^{\etahat} \}
\end{equation*}
be the labels of the toric divisors that contain $x$.
Consider the homomorphism
$$\prod_{\etahat \in I_x} S^1_{\etahat} \to T$$
whose restriction to each factor is its inclusion map into~$T$.
Because the $T$ action on $M$ is locally standard,
this homomorphism is one-to-one, and its image is $T_x$.
%\ynote{ \\ We refer to \Cref{Ix Tx} from \Cref{Ix Tx'}
%and from the proof of \Cref{free}.}
\eor
\end{remark}

\begin{remark} \labell{Ix Tx'}
Fix $\etahat \in \tZhat$, and let $x \in F := M^\etahat$.
Let $T_x$ and $I_x$ be as in \Cref{Ix Tx}.
Let $T_x'$ be the kernel of the $T_x$ action 
on the fibre at $x$ of the normal bundle $\nu_M(F)$.
Then $T_x'$ is the image in $T$ of 
$\prod\limits_{\etahat' \in I_x \setminus \{\etahat\}} S^1_{\etahat'}$.
%\ynote{ \\ We refer to \Cref{Ix Tx'} from \Cref{etahat torsor}
%and from the proof of \Cref{free}.}
\eor
\end{remark}

We begin with \emph{radial blowup} along an individual toric divisor.

\begin{remark}
Radial blowup along a submanifold was defined 
by Klaus J\"anich in \cite{janich}.
It defines a functor from the category of pairs $(M,N)$
where $M$ is a manifold and $N$ is a submanifold,
with arrows the local diffeomorphisms of $M$ that preserve $N$,
to the category of manifolds, with arrows the local diffeomorphisms.
For an individual toric divisor,
our smooth structure will differ from J\"anich's,
but it will agree with J\"anich's on the level of sets.
A similar construction for local circle actions was introduced 
by Karshon in \cite{cutting}.
\eor
\end{remark}

Fix a locally standard $T$ manifold $M$
and a toric divisor $F := M^\etahat$.
Let $\nu_M(F)$ be the normal bundle of $F$ in~$M$.  Then 
%\ynote{We refer to \eqref{SMF} from \Cref{P for model}}
\begin{equation} \labell{SMF}
 S_M(F) \, := \, (\nu_M(F) \ssminus 0)/\Rpos 
\end{equation}
is a $T$-equivariant principal $S^1_\etahat$-bundle over $F$,
with the bundle map to $F$ 
induced from the bundle map $\nu_M(F) \to F$.

\begin{construction}
\labell{c:radial blowup}
As a set, the \textbf{radial blowup} of $M$ along the toric divisor
$F := M^\etahat$ is the disjoint union
$$ M \odot F \, := \, 
   (M \ssminus F) \, \sqcup \, S_M(F) \, ,$$
equipped with the map
$$ c_\etahat \colon M \odot F \to M $$
whose restriction to $M \ssminus F$ is the inclusion map into $M$
and whose restriction to $S_M(F)$ is the bundle map $S_M(F) \to F$.
%\ynote{We refer to \Cref{c:radial blowup} from \Cref{P for model}.}
\eoc
\end{construction}

\begin{remark} 
The $T$-action on $M$ induces a $T$-action on $M \odot F$,
the map $c_\etahat \colon M \odot F \to M$ is $T$-equivariant,
and the composition 
$$ \xymatrix{
 M \odot F \ar[r]^{c_{\etahat}} 
   & M \ar[r]^{\pi} & Q 
} $$
is a quotient map
(namely, it induces a bijection from $(M \odot F)/T$ to $Q$).
\end{remark}

\begin{remark}
When the toric divisor $F$ is the empty set, we obtain 
$$ M \odot \emptyset = M ,$$
equipped with the identity map to $M$.
\eor
\end{remark}

\begin{remark} \labell{etahat torsor}
Let $x \in F$.  Then $S^1_{\etahat}$
acts freely and transitively on the preimage of $x$ in $M \odot F$.
Moreover, let $T_x$ be the stabilizer of $x$ in $T$,
and let $T_x'$ be as in \Cref{Ix Tx'}.
Then $T_x$ acts transitively on the preimage of $x$ in $M \odot F$,
through the projection $T_x \to S^1_\etahat$ with kernel $T_x'$.
%\ynote{We refer to \Cref{etahat torsor} from the proof of \Cref{free}.}
\eor
\end{remark}

\begin{construction}
\labell{c:radial blowup map}
Let $M'$ be another locally standard $T$-manifold.
Fix $\etahat \in \tZhat$, and let $F = M^\etahat$ and $F' = (M')^\etahat$.
Then any $T$-equivariant diffeomorphism
$\psi \colon M \to M'$ induces a $T$-equivariant bijection
$\tilde{\psi} \colon M \odot F \to M' \odot F'$,
inducing a pullback diagram (of $T$-equivariant bijections)
$$ \xymatrix{
 M \odot F \ar[r]^{\tilde{\psi}} \ar[d]_{c_\etahat} 
     & M' \odot F' \ar[d]^{c_\etahat} \\
 M \ar[r]^{\psi} & M' \, .
}$$
%\ynote{We refer to \Cref{c:radial blowup map} from \Cref{c:psi tilde}.}
\eoc
\end{construction}

\begin{remark} \labell{r:radial blowup map}, 
In \Cref{c:radial blowup map}, 
if $M=M'$ and $\psi = \Id_M$, then $\tilde{\psi} = \Id_{M \odot F}$, 
and if $\psi = \psi_2 \circ \psi_1$, 
then $\tilde{\psi} = \tilde{\psi_2} \circ \tilde{\psi_1}$. 
%\ynote{We refer to \Cref{r:radial blowup map} from \Cref{r:psi tilde}.}
\eor
\end{remark}

We now perform the radial blowup
\emph{simultaneously} along all the toric divisors, as follows.

\begin{construction} \labell{c:simultaneous}
As a set, the \textbf{simultaneous toric radial blowup} $P_M$ 
is the fibred product over $M$
of the radial blowups along all the toric divisors.
Namely, let $\{ F_\etahat \}_{\etahat \in \tZhat}$ 
be the set of toric divisors,
and consider the corresponding radial blowups
and maps $c_{\etahat} \colon M \odot F_{\etahat} \to M$.  
Then we take 
$$ P_M := \Big\{ (y_{\etahat})_{\etahat \in \tZhat}
   \in \prod_{\etahat \in \tZhat} (M \odot F_{\etahat}) \ \, | \, \ 
   \text{ there exists $x \in M$ such that } 
   c_{\etahat}(y_{\etahat}) = x \text{ for all $\etahat$ } \Big\} ,$$
equipped with the map 
$$ c_{\odot} \colon P_M \to M \quad , \quad 
   (y_{\etahat})_{\etahat \in \tZhat} \mapsto x \ 
\text{ such that } 
   c_{\etahat}(y_{\etahat}) = x \text{ for all $\etahat$\,.}$$ 
%\ynote{We refer to \Cref{c:simultaneous} from the proof of \Cref{free},
%\Cref{c:psi tilde}, and \Cref{P for model}.}
\eoc
\end{construction}

\begin{lemma} \labell{free}
The $T$-action on $M$ induces a $T$-action on $P_M$,
and the map $ c_{\odot} \colon P_M \to M $ is $T$-equivariant.
Moreover,
the $T$-action on $P_M$ is free, and the composition
$$ \xymatrix{
P_M \ar[rr]^{c_{\odot}} \ar@/_1pc/[rrrr]_{\Pi} && M \ar[rr]^{\pi} && Q 
}$$ 
is a quotient map (namely, it induces a bijection 
from $P_M/T$ to $Q$).
%\ynote{ \\ We refer to \Cref{free} earlier in this section.}
\end{lemma}

\begin{proof}
It is enough to show that, for each $x \in M$, its stabilizer $T_x$
acts freely and transitively on the preimage of $x$ in $P_M$.
This, in turn, follow from \Cref{c:simultaneous},
\Cref{etahat torsor}, and \Cref{Ix Tx,Ix Tx'}.
\end{proof}

\begin{construction} \labell{c:psi tilde}
Let $\psi \colon M \to M'$ be a $T$-equivariant diffeomorphism.
By Constructions \ref{c:radial blowup map} and \ref{c:simultaneous},
$\psi$ induces a $T$-equivariant map
from the simultaneous toric radial blowup $P_{M}$ of $M$
to the simultaneous toric radial blowup $P_{M'}$ of $M'$,
inducing a pullback diagram (of $T$-equivariant bijections)
$$ \xymatrix{
 P_{M} \ar[rr]^{\tilde{\psi}} \ar[d]_{c_{\odot}} 
  && P_{M'} \ar[d]^{c_{\odot}} \\
 M \ar[rr]^{\psi} && M' \, .
} $$
%\ynote{We refer to \Cref{c:psi tilde} from 
%\Cref{r:psi tilde} and \Cref{crucial cor followup}}
\eoc
\end{construction}

\begin{remark} \labell{r:psi tilde}
In \Cref{c:psi tilde},
if $\psi=\Id_M$, then $\tilde{\psi} = \Id_{P_M}$,
and if $\psi = \psi_2 \circ \psi_1$ 
then $\tilde{\psi} = \tilde{\psi_2} \circ \tilde{\psi_1}$.
This follows from \Cref{r:radial blowup map}.
%\ynote{We refer to \Cref{r:psi tilde} earlier in this section.}
\eor
\end{remark}

\begin{remark} \labell{open in blowup}
Our construction commutes with inclusions of open subsets,
in the following sense.
Let $M$ be a locally standard $T$-manifold, and let $U_M$ be
a $T$ invariant open subset.
Then the set $P_{U_M}$ 
that is obtained by applying our functor to $U_M$
is a $T$-invariant subset of the set $P_M$
that is obtained by applying our functor to $M$,
the map $c_{\odot} \colon P_{U_M} \to U_M$
is the restriction of the map $c_{\odot} \colon P_M \to M$,
and the inclusion map $U_M \to M$ gets mapped to 
the inclusion map $P_{U_M} \to P_M$.
Thus, we obtain the following commuting diagram,
where the horizontal maps are inclusions of subsets:
$$ \xymatrix{
 P_{U_M} \ar[d]_{c_{\odot}} \ar@/_2pc/[dd]_{\Pi}
  & \hookrightarrow & P_M \ar[d]^{c_{\odot}} \ar@/^2pc/[dd]^{\Pi} \\ 
  U_M \ar[d]_{\pi} & \hookrightarrow & M \ar[d]^{\pi} \\
  U_Q & \hookrightarrow & Q \, .
}$$
%\ynote{We refer to \Cref{open in blowup} from the proofs of \Cref{xi xi'}
%and of \Cref{natural}.}
\eor
\end{remark}

%-------------------------------------------
\section{Diffeomorphisms of models lift: from $M$ to $P$}
\labell{sec:M to P}
%-------------------------------------------

In this section we work purely within models.

Recall that, for any non-negative integers $n,l,m$,
we have the model cutting map $c^\std$ 
and the model quotient map $\theta$
$$ \xymatrix{
\Rplus^n \times \R^m \times \T^{n+l}
 \ar[rr]^{c^{\std}}
   && \C^n \times \T^l \times \R^m
 \ar[rr]^{\theta}
   && \Rplus^n \times \R^m , 
}$$
\begin{multline*}
 (s_1,\ldots,s_n;x_1,\ldots,x_m;b_1,\ldots,b_{n+l})
 \mapsto (z_1,\ldots,z_n;b_{n+1},\ldots,b_{n+l};x_1,\ldots,x_m)  \\
 \mapsto (s_1,\ldots,s_n;x_1,\ldots,x_m) , 
\end{multline*}
given, respectively, by $z_j = \sqrt{s_j}b_j$ and $s_j = |z_j|^2$
for $j = 1,\ldots,n$.

We pause to set up some notation.
For any $0 \leq k \leq n$,
we identify $\C^k \times (\Cx)^{n-k}$ and $\R^k \times \Rpos^{n-k}$
with their images in $\C^n$ and $\R^n$,
(namely, we omit their parenthetization).
Similarly, we identify $\T^k \times \T^{n-k}$ with $\T^n$.
For $d=n+l$, we identify $\T^n \times \T^l$ with $\T^d$.
For a $\T^d$-invariant open subset
$ \Omega \subset (\C^k \times (\Cx)^{n-k}) \times \T^l \times \R^m$,
we write its coordinates as $(z,c,x),$ with 
\begin{multline*}
 z = (z_1,\ldots,z_k,z_{k+1},\ldots,z_n) \in \C^k \times (\Cx)^{n-k}, 
\qquad
 c = (c_{n+1},\ldots,c_d) \in \T^{l}, \\
\text{ and } \qquad
 x = (x_1,\ldots,x_m) \in \R^m.
\end{multline*}
For an open subset $\calO \subset (\Rplus^{k} \times \Rpos^{n-k}) \times \R^m$,
we write the coordinates of the product $\calO \times \T^d$ as $((s,x),(a,b))$, 
with 
\begin{multline*}
s = (s_1,\ldots,s_k,s_{k+1},\ldots,s_n) \in \Rplus^k \times \Rpos^{n-k},
\qquad
x = (x_1,\ldots,x_m) \in \R^m, \\
a = (a_1,\ldots,a_k) \in \T^k, 
\quad \text{ and }  \qquad
 b = (b_{k+1},\ldots,b_d) \in \T^{d-k}.
\end{multline*}

\begin{construction} \labell{P for model}
For $d=n+l$, let 
$ \Omega \subset \C^n \times \T^l \times \R^m$
be a $\T^d$ invariant subset,
and let 
$\calO \subset \Rplus^{n} \times \R^m$
be its image under the map $\theta$.
We will now construct a $\T^d$-equivariant bijection
$$ \xi_\Omega \colon P_\Omega \to \calO \times \T^d,$$
where $P_\Omega$ is obtained from $\Omega$
by the set-theoretic simultaneous toric radial blowup
(\Cref{c:simultaneous}).
For each $1 \leq j \leq n$, let $F_j := \Omega \cap \{ z_j = 0 \}$
(which might be empty).
We identify the normal bundle $\nu_\Omega(F_j)$ 
of $F_j$ in $\Omega$ with $F_j \times \C$
by fixing the coordinates other than $z_j$
and identifying $T_0\C_{j\text{th}}$ with $\C$.
Identifying $\C^\times/\Rpos$ with $S^1$ by $z \mapsto \dfrac{z}{|z|}$,
we obtain an identification of the circle bundle $S_\Omega(F_j)$
(see \eqref{SMF})
with $F_j \times S^1$. 
This allows us to identify the radial blowup $\Omega \odot F_j$ 
along $F_j$ (\Cref{c:radial blowup})
with the preimage of~$\calO$ under the map 
$$(\C^{j-1} \times (\Rplus \times S^1) \times \C^{n-j}) 
\times \T^l \times \R^m \to \Rplus^n \times \R^m
$$
that is given by
\begin{multline*}
(z_1,\ldots,z_{j-1},(s_j,b_j),z_{j+1},\ldots,z_n;
    b_{n+1},\ldots,b_{n+l};x_1,\ldots,x_m)
 \mapsto  \\
  (|z_1|^2,\ldots,|z_{j-1}|^2,s_j,|z_{j+1}|^2,\ldots,|z_n|^2; x_1,\ldots,x_m) .
\end{multline*}
Doing this for all $1 \leq j \leq n$,
we obtain our equivariant bijection 
$\xi_\Omega \colon P_\Omega \to \calO \times \T^d$.
%\ynote{ \\ We refer to \Cref{P for model} from \Cref{cute diagram},
%\Cref{crucial cor} and its proof, \Cref{crucial cor followup},
%and the proof of \Cref{xi xi'}.}
\eoc
\end{construction}

\begin{remark} \labell{cute diagram}
Let $\Omega$ be a $\T^d$-invariant open subset of 
$ \C^n \times \T^l \times \R^m $,
let $P_{\Omega}$ be its simultaneous toric radial blowup,
and let $\calO := \theta(\Omega) \subset \Rplus^n \times \R^m $.
The bijection 
$\xi_\Omega \colon P_\Omega \to \calO \times \T^d$
of \Cref{P for model} fits into a commuting diagram
$$
\xymatrix{
 P_{\Omega} \ar[rr]_{\cong}^{\xi_\Omega} \ar[dr]_{c_{\odot}} 
                  \ar[dd]_{\Pi}
 && \calO \times \T^d \ar[dl]^{c^\std}
 \ar[dd]^{\text{projection}}
 \\
  & \Omega \ar[dr]^{\theta^\std} \ar[ld]_{\pi} & \\ 
\Omega/\T^d \ar[rr]^{\cong} && \calO \, ,
}
$$
where $c^\std$ and $\theta^\std$ are the model cutting map
and the model quotient map.
%\ynote{ \\ We refer 
%to \Cref{cute diagram} from \Cref{xi} and the proof of \Cref{natural2}.}
\end{remark}

\begin{proposition} \labell{characterization}
Fix non-negative integers $n$, $n'$, $l$, $l'$, $m$, $m'$, 
such that $n+l=n'+l'=:d$ 
and a non-negative integer $k$ such that $k \leq \min\{n,n'\}$.
Let $\Omega$ be a $\T^d$-invariant open subset of 
$(\C^k \times (\Cx)^{n-k}) \times \T^l \times \R^m$,
and let $\calO$ be its image in $(\Rplus^k \times \Rpos^{n-k}) \times \R^m$.
Let 
$$ \rho \colon \T^d \to \T^d$$ 
be an isomorphism that fixes $\T^k \times \{1\}^{d-k}$.  
Let 
$$ \rho_j \colon \T^{d-k} \to S^1 , \qquad \text{ for } 1 \leq j \leq d ,$$
be the restriction of its components
to $\{1\}^k \times \T^{d-k}$. 

Then a function 
$$ f \colon \Omega \to 
   (\C^k \times (\Cx)^{n'-k}) \times \T^{l'} \times \R^{m'} $$
is smooth and $\rho$-equivariant if and only if
there exist smooth functions
$$A_j = A_j(s,x) \colon \calO \to \C \qquad \text{ for } 1 \leq j \leq k , $$
$$A_j = A_j(s,x) \colon \calO \to \Cx \qquad \text{ for } k < j \leq n' , $$
$$A_j = A_j(s,x) \colon \calO \to S^1 \qquad \text{ for } n' < j \leq d , 
$$
and 
$$x_j' = x_j'(s,x) \colon \calO \to \R \qquad \text{ for } 1 \leq j \leq m', $$
such that, setting
\begin{equation} \labell{s and b}
   s = (|z_1|^2,\ldots,|z_n|^2) \qquad \text{ and } \qquad
   b = \big( \frac{z_{k+1}}{|z_{k+1}|}, \ldots, \frac{z_n}{|z_n|}, 
             c_{n+1},\ldots,c_{d} \big) ,
\end{equation}
the first $k$ components 
of $f(z_1,\ldots,z_n;c_{n+1},\ldots,c_{d};x_1,\ldots,x_m)$ are
$$ z_j' = z_j \rho_j(b) A_j(s,x) \qquad \text{ for } 1 \leq j \leq k,$$
its next $n'-k$ components are
$$ z_j' = \rho_j(b) A_j(s,x) \qquad \text{ for } k < j \leq n',$$
its next $l'$ components are
$$ c_j' = \rho_j(b) A_j(s,x) \qquad \text{ for } n' < j \leq d,$$
and its last $m'$ components are
$$ x_j' = x_j'(s,x) \qquad \text{ for } 1 \leq j \leq m'.$$
\end{proposition}

\begin{proof}
First, note that 
$$ \rho(a,b) = 
   (a_1\rho_1(b),\ldots,a_k\rho_k(b);\rho_{k+1}(b),\ldots,\rho_d(b))$$ 
for all $a \in \T^k$ and $b \in \T^{n-k}$,
and note that if there exist smooth functions 
$A_1,\ldots,A_d$, $x_1',\ldots,x_m'$ as above
then $f$ is smooth and $\rho$-equivariant.

From now on, assume that $f$ is smooth and $\rho$-equivariant.

The last $m$ coordinates of $f$ are real-valued smooth functions
of $(z,c,x)$ and are $\T^d$ invariant.
By \Cref{thetabar is diffeo},
they can be written in the required form $x_j'(s,x)$.

For $n' < j \leq d$, 
the product of the $S^1$ valued smooth function $\rho_j(b)^{-1}$
with the corresponding coordinate $c_j'$ of $f$
is an $S^1$-valued smooth function of $(z,c,x)$ that is $\T^d$ invariant.
By \Cref{thetabar is diffeo},
it can be written in the required form $A_j(s,x)$.

For $k < j \leq n'$, 
the product of the $S^1$ valued smooth function $\rho_j(b)^{-1}$
with the corresponding coordinate $z_j'$ of $f$
is a non-vanishing complex-valued smooth function 
of $(z,c,x)$ that is $\T^d$ invariant.
By \Cref{thetabar is diffeo},
it can be written in the required form $A_j(s,x)$.

To show the first $k$ components of $f$ have the required form,
we apply ``Bredon's trick'' (see \cite[Chap.~VI, \S 5]{Br}). 
Fix $1 \leq j \leq k$.  
Write $(z,c,x) = (z_j,v)$ with $v$ consisting of the coordinates
other than $z_j$. 
The $j$th coordinate of $f$ then becomes a complex-valued function 
$f_j(z_j,v)$ on $\Omega$
that satisfies $af_j(z_j,v) = f_j(az_j,v)$ for all $a \in S^1$.
Restricting to real values of $z_j$ and to $a=-1$, 
we obtain a complex-valued smooth function $f_j(r,v)$ 
defined for all $(r,v)$ in $\Omega$ with $r$ real
that satisfies
$f_j(-r,v) = -f_j(r,v)$ for all such $(r,v)$.
In particular, $f_j(0,v) = 0$.
By Hadamard's lemma (in one variable, with parameters; 
see \Cref{fn:hadamard} on p.~\pageref{fn:hadamard}),
there exists a smooth function $A_j'(r,v)$ such that
$f_j(r,v) = r A_j'(r,v)$ for all such $(r,v)$.
Because $f_j(-r,v) = -f_j(r,v)$,
we have $A_j'(-r,v) = A_j'(r,v)$ for all such $(r,v)$.
By Whitney \cite{whitney}
(see \Cref{thetabar is diffeo}),
there exists a smooth function $B_j(s,v)$
such that $A_j'(r,v) = B_j(r^2,v)$ for all such $(r,v)$.
For all $z_j \neq 0$,
\begin{multline*}
f_j(z_j,v) = \frac{z_j}{|z_j|} f_j(|z_j|,v) 
\qquad \text{ by $S^1$-equivariance } \\ 
 \ = \ \frac{z_j}{|z_j|} |z_j| A_j'(|z_j|,v)
\qquad \text{ by the choice of $A_j'$ }   \\
 \ = \ z_j B_j(|z_j|^2,v)
\qquad \text{ by the choice of $B_j$ }.
\end{multline*}
Because this function 
intertwines the $\T^d$-action on $\Omega$
with its action on $\C$ that is given by the homomorphism 
$(a,b) \mapsto a_j \rho_j(b)$,
the complex-valued smooth function 
$(z_j,v) \mapsto \rho_j(b)^{-1} B_j(|z_j|^2,v)$
is $\T^d$ invariant when $z_j \neq 0$. 
By continuity, it is $\T^d$ invariant everywhere on $\Omega$.
By Whitney \cite{whitney} (see \Cref{thetabar is diffeo}),
we can write it as $A_j(s,x)$ with $A_j$ smooth.
So 
$$ f_j(z_j,v) = z_j \rho_j(b) A_j(s,x) $$
whenever $|z_j| \neq 0$, hence everywhere on $\Omega$,
\end{proof}

\begin{corollary} \labell{crucial cor}
Fix non-negative integers $n,n',l,l',m,m'$ such that $n+l=n'+l'=d$
and a non-negative integer $k$ such that $k \leq n$ and $k \leq n'$.
Fix an isomorphism 
$$ \rho \colon \T^d \to \T^d$$
that fixes $\T^k \times \{ 1 \}^{d-k}$.
Let 
$$ f \colon \Omega \to \Omega '$$
be a $\rho$-equivariant diffeomorphism 
from an invariant open subset 
$\Omega$ of $(\C^k \times (\Cx)^{n-k} ) \times \T^{l} \times \R^{m}$
to an invariant open subset 
$\Omega'$ of $(\C^k \times (\Cx)^{n'-k} ) \times \T^{l'} \times \R^{m'}$.
Denote by $\calO = \theta(\Omega)$ and $\calO'=\theta(\Omega')$ 
the corresponding open subsets 
of $(\Rplus^k \times \Rpos^{n-k}) \times \R^{m}$
and of $(\Rplus^k \times \Rpos^{n'-k}) \times \R^{m'}$.
Let
$$ G \colon \calO \times \T^d \to \calO' \times \T^d$$
be the map that is obtained 
from the map 
$\tilde{f} \colon P_\Omega \to P_{\Omega'}$
that $f$ induces 
on the simultaneous toric radial blowups of $\Omega$ and of $\Omega'$
(\Cref{c:psi tilde})
by identifying these simultaneous toric radial blowups 
with $\calO \times \T^d$ and $\calO' \times \T^d$
as in Construction~\ref{P for model}.
Then $G$ is a diffeomorphism.
%\ynote{We refer to \Cref{crucial cor} from the proof
%of \Cref{crucial cor followup}.}
\end{corollary}

\begin{proof}
Let $A_j(s,x)$ and $x_j'(s,x)$ be smooth functions as in 
Proposition~\ref{characterization}.
We show below that the map $G$ is given by 
$$ G \colon ((s,x),(a,b)) \mapsto ((s',x'),(a',b'))$$
where
$$ s'_j = s_j |A_j(s,x)|^2 
\quad \text{ and } \quad
   a'_j = a_j \rho_j(b) \frac{A_j(s,x)}{|A_j(s,x)|} 
\qquad   \text{ for } \quad 1 \leq j \leq k , $$
where
$$ s'_j = |A_j(s,x)|^2 
\quad \text{ and } \quad
   b'_j = \rho_j(b) \frac{A_j(s,x)}{|A_j(s,x)|} 
\qquad \text{ for } \quad k < j \leq n' , $$
where
$$ b'_j = \rho_j(b) A_j(s,x)
\qquad \text{ for } \quad n' < j \leq d , $$
and where
$$ x'_j = x_j(s,x) \qquad \text{ for } \quad 1 \leq j \leq m' .$$
In particular, the map $G$ is smooth. 
Applying the same argument to $f^{-1}$, we conclude
that the map $G$ is a diffeomorphism.

For each $1 \leq j \leq k$, the corresponding toric divisors are
$F_j = \Omega \cap \{ z_j = 0 \}$, 
and $F'_j = \Omega' \cap \{ z_j' = 0 \}$.
Since $f$ is a $\rho$-equivariant diffeomorphism 
and $\rho$ fixes $\T^k \times \{1\}^{d-k}$, we have $F_j = f^{\pre}(F'_j)$.

We identify the normal bundle of $F_j$ in $\Omega$ with $F_j \times \C$
by fixing the coordinates other than $z_j$
and identifying $T_0\C_{j\text{th}}$ with $\C$,
and similarly for $F'_j$.
With these identifications,
at each point in $F_j$,
the \emph{vertical differential}
 --- defined as the map that $df$ induces on the fibre of the normal bundle
 --- 
is multiplication by $\rho_j(b) A_j(s,x)$ for the $s,b$
as in \eqref{s and b}. 
(Note that $F_j = \{s_j=0\}$.)
Since $f$ is a diffeomorphism, its component $f_j \colon \Omega \to \C$
is a submersion, and since $f_j$ vanishes along $F_j$,
its vertical differential must be invertible on the fibre of the normal bundle
of $F_j$ at each point of $F_j$.

We conclude that $A_j(s,x) \neq 0$ not only for $k < j \leq n'$,
but also for $1 \leq j \leq k$.
Indeed, outside $F_j$ this follows from 
$f^{\pre}(\Omega' \ssminus F_j') = \Omega \ssminus F_j$,
and along $F_j$ this follows from the invertibility of the 
vertical differential.

Identifying $S_M(F_j) := \nu_M(F_j) / \Rpos$ with $F_j \times S^1$,
the vertical differential of $f$ induces the map $S_M(F_j) \to S_M(F_j')$ 
(which lifts the map $f|_{F_j} \colon F_j \to F_j'$, and )
whose fibre component at each point of $F_j$
is multiplication by $\rho_j(b) \frac{A_j(s,x)}{|A_j(s,x)|}$.
(Here, $z_j = 0$ and $s_j=0$.)

Construction~\ref{P for model} 
identifies the radial blowup $\Omega \odot F_j$ 
with the preimage of $\calO$ in 
$(\C^{j-1} \times (\Rplus \times S^1) \times \C^{k-j}
 \times (\Cx)^{n-k}) \times \T^l \times \R^m$.
Doing this for all $1 \leq j \leq k$,
and identifying each $\Cx$ with $\Rpos \times S^1$
by $z \mapsto (|z|^2,\frac{z}{|z|})$,
we obtain an identification of the simultaneous toric radial blowup 
$P_\Omega$ with $\calO \times \T^d$. 
We similarly identify the simultaneous toric radial blowup $P_{\Omega'}$ 
with $\calO' \times \T^d$.
With these identifications, the map that $f$ induces
becomes the map $G$ of the form described above. 
%\ynote{details?}
\end{proof}

Note that the following diagram commutes:
$$ \xymatrix{
   \calO \times \T^d
       \ar[d]^{c^{\std}} \ar[rr]^(.6){G \qquad} && \calO' \times \T^d
              \ar[d]^{{c}^{\std}}  \\
\Omega \ar[d]^{\theta} \ar[rr]^(.6){f \qquad } && \Omega' \ar[d]^{\theta} \\
 \calO \ar[rr]^{g} && \calO' .
} $$
Here, $f$ is a given $\rho$-equivariant diffeomorphism, 
$G$ is the diffeomorphism from Corollary \ref{crucial cor},
and $g$ is the diffeomorphism from \Cref{f to g}.

\begin{proposition} \labell{crucial cor followup}
Fix non-negative integers $n, n', l, l', m, m'$ such that $n+l=n'+l'=:d$.
Let $\rho \colon \T^d \to \T^d$ be an isomorphism.
Let 
$$ f \colon \Omega \to \Omega' $$
be a $\rho$-equivariant diffeomorphism
from an invariant open subset $\Omega$ of $\C^n \times \T^l \times \R^m$
to an invariant open subset $\Omega'$ 
of $\C^{n'} \times \T^{l'} \times \R^{m'}$.
Let $\calO$ and $\calO'$ be the images of $\Omega$ and $\Omega'$
in $\Rplus^n \times \R^m$ and in $\Rplus^{n'} \times \R^{m'}$.
Let
$$ G \colon \calO \times \T^d \to \calO' \times \T^d $$
be the map that is obtained from the map 
$\tilde{f} \colon P_\Omega \to P_{\Omega'}$
that $f$ induces
on the simultaneous toric radial blowups of $\Omega$ and of $\Omega'$
(\Cref{c:psi tilde})
by identifying these simultaneous toric radial blowups 
with $\calO \times \T^d$ and $\calO' \times \T^d$ as in \Cref{P for model}.
Then $G$ is a diffeomorphism.
%\ynote{ \\ 
%We refer to \Cref{crucial cor followup} from the proof of \Cref{xi xi'}.}
\end{proposition} 

\begin{proof}
After possibly permuting the coordinates of $\C^n$ in the first model
and of $\C^{n'}$ in the second model,
and after possibly replacing some of the coordinates of $\C^{n'}$
in the second model by their complex conjugates,
and after adjusting $\rho$ accordingly,
we may assume that there is an integer $1 \leq k \leq n$
such that the set $\calO$ 
meets the first $k$ facets of the model $\Rplus^n \times \R^m$
and does not meet any of its other facets,
the set $\calO'$
meets the first $k$ facets of the model $\Rplus^{n'} \times \R^{m'}$
and does not meet any of its other facets,
and $\rho \colon \T^d \to \T^d$ fixes $\T^k \times \{1\}^{d-k}$.
The result of Proposition~\ref{crucial cor followup}
then follows from Corollary~\ref{crucial cor}. 
\end{proof}

%-----------------------------------------------------------------------
\section{The simultaneous toric radial blowup functor, smoothly}
\labell{sec:blowup smoothly}

The quotient functor $\frakM_\iso \to \frakQ_\iso$ (\Cref{def:functor})
requires us to work with a fixed torus, $T$.
We would like to apply this functor to locally standard $T$-charts.
For this, we need to work not only with our fixed torus $T$,
but also with the standard torus $\T^d$, 
and we need to allow different isomorphisms $T \to \T^d$.
We formalize this in the following remark.

\begin{remark} \labell{T and TT}
Let $M$ be a locally standard $T$-manifold,
and let 
$$ \phi \colon U_M \to \Omega $$
be a locally standard $T$-chart, 
equivariant with respect to an isomorphism
$$ \rho \colon T \to \T^d $$ (\Cref{def:locally standard}),
with $\Omega$ an open subset of $\C^n \times \T^l \times \R^m$
for non-negative integers $n,l,m$ with $n+l=d$.
We can now consider $\Omega$ with the $T$-action
that is obtained as the composition of the isomorphism 
$\rho \colon T \to \T^d$
with the given $\T^d$-action.
Then, both $U_M$ and $\Omega$ have $T$-actions,
and the map $\phi$ is $T$-equivariant,
so we can apply to it the simultaneous toric radial blowup functor
(Constructions \ref{c:simultaneous} and \ref{c:psi tilde}).
This yields $T$-sets $P_{U_M}$ and $P_{\Omega}$
and a $T$-equivariant bijection 
$\tilde{\phi} \colon P_{U_M} \to P_{\Omega}$,
such that the following diagram commutes.
$$ \xymatrix{
 P_{U_M} \ar[r]^{\tilde{\phi}} \ar[d]_{c_{\odot}} 
   & P_{\Omega} \ar[d]^{c_{\odot}} \\ 
 U_M \ar[r]^{\phi} & \Omega \, .
}$$
If we apply the radial toric blowup functor to $\Omega$
with its given $\T^d$-action, we obtain the same set $P_{\Omega}$,
with a $\T^d$-action. 
The $T$-action on $P_{\Omega}$
coincides with the composition of~$\rho$
with the $\T^d$-action on $P_{\Omega}$.
As a map from a $T$-manifold to a $\T^d$-manifold,
the map $\tilde{\phi}$ is $\rho$-equivariant.
\eor
\end{remark}

\Cref{T and TT} follows from the following observation.

\begin{remark} \labell{category of pairs}
Fix a positive integer $d$.
Consider the category of pairs $(T,M)$
where $T$ is a $d$ dimensional torus
and $M$ is a locally standard $T$-manifold,
and where a morphism $(T,M) \to (T',M)$ is a pair
$(\rho,\psi)$ where $\rho \colon T \to T'$ is an isomorphism
and $\psi \colon M \to M'$ is a $\rho$-equivariant diffeomorphism.
The simultaneous toric radial blowup construction 
(Constructions \ref{c:simultaneous} and \ref{c:psi tilde})
yields a functor from this category to the category of pairs $(T,X)$,
where $T$ is a $d$ dimensional torus and $X$ is a set with a free $T$-action,
where a morphism $(T,X) \to (T',X')$ is a pair $(\rho,\tpsi)$
with $\rho \colon T \to T'$ an isomorphism and $\tpsi \colon X \to X'$
a $\rho$-equivariant map of sets.
%\ynote{
 %\\ We use this remark -- in the special case of charts -- 
%to prove \Cref{T and  TT}.}
\eor
\end{remark}

\begin{construction} \labell{xi}
Let $M$ be a locally standard $T$ manifold, 
and let 
$$\phi \colon U_M \to \Omega$$ 
be a locally standard $T$-chart,
equivariant with respect to an isomorphism $\rho \colon T \to \T^d$,
with $\Omega$ an open subset of $\C^n \times \T^l \times \R^m$
for non-negative integers $n,l,m$ with $n+l=d$.
By \Cref{T and TT}, 
%\ynote{check},
the simultaneous toric radial blowup functor 
yields a $\rho$-equivariant bijection 
$ \tilde{\phi} \colon P_{U_M} \to P_{\Omega} $
between these set-theoretic simultaneous toric radial blowups,
such that the following diagram commutes,
where $\calO := \theta^\std (\Omega) \subset \Rplus^n \times \R^m$.
$$ \xymatrix{
 P_{U_M} \ar[r]^{\tilde{\phi}} \ar[d]_{c_\odot} 
 \ar@/_2pc/[dd]_{\Pi}
 & P_\Omega \ar[d]^{c_{\odot}} \\
 U_M \ar[r]^{\phi} \ar[d]_{\pi} & \Omega \ar[d]^{\theta^\std} \\ 
 U_Q \ar[r]^{\psi} & \calO \, .
}$$
Composing with the map $\xi_\Omega \colon P_\Omega \to \calO \times \T^d$
of \Cref{cute diagram}, we obtain a $\rho$-equivariant bijection 
$ \xi \colon P_{U_M} \to \calO \times \T^d $
such that the following diagram commutes,
where $c^\std$ and $\theta^\std$ are the model cutting map 
and the model quotient map.
{
$$
\xymatrix{
 P_{U_M} \ar[d]_{c_{\odot}} \ar@/_2pc/[dd]_{\Pi} \ar[r]^{\xi}_{\cong}
 & **[r]{\calO \times \T^d} 
   \ar[d]^{c^\std} \ar@/^2.5pc/[dd]^{\text{projection}} \\
 U_M \ar[r]^{\phi} \ar[d]_{\pi} & \Omega \ar[d]^{\theta^\std} \\
 U_Q \ar[r]^{\psi} & \calO
}
$$
}
\eoc
\end{construction}

\begin{lemma} \labell{xi xi'}
Fix a locally standard $T$-manifold $M$.
For any two bijections $\xi$ and $\xi'$ that 
arise as in Construction~\ref{xi},
the composition $\xi' \circ \xi^{-1}$ is a diffeomorphism
between invariant open subsets of $\Rplus^n \times \R^m \times \T^d$
and $\Rplus^{n'} \times \R^{m'} \times \T^d$
for the corresponding non-negative integers $n,m,n',m',d$.
\end{lemma}

\begin{proof}
%\ynote{check:}
Let 
$$\xi \colon P_{U_M} \to \calO \times \T^d
\quad \text{ and } \quad
  \xi' \colon P_{U'_M} \to \calO' \times \T^d$$
be two bijections that arise as in Construction~\ref{xi}.
By \Cref{open in blowup},
${P}_{U_M} \cap {P}_{U'_M} = {P}_{U_M \cap U'_M}$.
After replacing $U_M$ and $U'_M$ by $U_M \cap U'_M$,
we may assume that $U_M = U'_M$, 
and we obtain a $\rho$-equivariant bijection 
$$\xi' \circ \xi^{-1} \colon
\calO \times \T^d \to \calO' \times \T^d$$
for an isomorphism $ \rho \colon \T^d \to \T^d $.
Chasing the definitions, we obtain a diagram
$$
\xymatrix{
 \calO \times \T^d \ar[rr]^{ \xi' \circ \xi^{-1} } && \calO' \times \T^d \\
 P_{\Omega} \ar[u] \ar[rr]^{\tphi'\ \circ \tphi^{-1} } && P_{\Omega'} \ar[u] 
}
$$
whose upward arrows are the bijections
that are obtained from \Cref{P for model}
and where $\tphi$ and $\tphi'$ are bijections 
$$\tilde{\phi} \colon P_{U_M} \to P_{\Omega} \quad \text{ and } \quad 
  \tilde{\phi}' \colon P_{U_M} \to P_{\Omega'}$$
that are obtained from locally standard $T$-charts with the same domain 
$$\phi \colon U_M \to \Omega \quad \text{ and } \quad 
  \phi' \colon U_M \to \Omega'$$
by applying the simultaneous toric radial blowup.
By functoriality of the set-theoretic 
simultaneous toric radial blowup (\Cref{r:psi tilde}),
the composition $\tilde{\phi}' \circ \tilde{\phi}^{-1}$
coincides with the map $\tilde{f}$ that is obtained
by applying the simultaneous toric radial blowup to the 
$\rho$-equivariant diffeomorphism
$$ f := \phi' \circ {\phi}^{-1} \colon \Omega \to \Omega'.$$
By Proposition~\ref{crucial cor followup}
with these maps $f$ and $\tilde{f}$, 
the bijection $G := \xi' \circ \xi^{-1}$ is a diffeomorphism.
\end{proof}

\begin{corollary} \labell{cor:structure}
Fix a locally standard $T$-manifold $M$,
let $P_M$ be its set theoretic simultaneous toric radial blowup,
and let $\Pi \colon P_{M} \to Q := M/T$ be the corresponding map
(\Cref{c:simultaneous}, \Cref{free}).
There exists a unique manifold-with-corners structure 
(in particular, a topology) on $P_{M}$ such that, 
for each locally standard $T$-chart $\phi \colon U_M \to \Omega_M$, 
the bijection $\xi \colon P_{U_M} \to \calO \times \T^d$ 
of Construction~\ref{xi} is a diffeomorphisms.
Moreover, with this structure, the map 
$$ \Pi \colon P_M \to Q $$
is a principal $T$-bundle.
\end{corollary}

\begin{proof}
By \Cref{locally standard criterion},
we can cover $Q$ by open sets $U_Q$
whose preimages $U_M$ in $M$ are domains of locally standard $T$-charts.
By \Cref{xi}, for each such $U_Q$, we obtain an equivariant bijection 
$\xi \colon P_M\big|_{U_Q} \to \calO \times \T^d$
where $\calO$ is an open subset of some $\Rplus^n \times \R^m$.
By \Cref{xi xi'}, the compositions $\xi' \circ \xi^{-1}$ are smooth.
By \Cref{mfld lemma2}, we obtain a manifold-with-corners structure on $P_M$.
The commuting diagrams
$$ \xymatrix{
 P_M\big|_{U_Q} \ar[d]^{\Pi} \ar[rr]^{\xi}
 && \calO \times \T^d \ar[d]^{\text{projection}} \\
 U_Q \ar[rr]^{\phi} && \calO,
} $$
where $\phi$ is the diffeomorphism
that is induced from the locally standard $T$-chart on $U_M$,
give local trivializations of $\Pi \colon P_M \to Q$.
\end{proof}

\begin{lemma} \labell{lift of diffeo} \ 
Let 
$$ f \colon M \to M' $$ 
be a $T$-equivariant diffeomorphism
of locally standard $T$-manifolds $M$ and $M'$.
Let 
$$\tilde{f} \colon P_{M} \to P_{M'}$$
be the induced map on their simultaneous toric radial blowups
(\Cref{c:psi tilde}).
Equip $P_M$ and $P_{M'}$ with their manifold-with-corners
structures of \Cref{cor:structure}.
Then $\tilde{f}$ is a diffeomorphism.
\end{lemma}

\begin{proof}
Let 
$$ \phi \colon U_M \to \Omega $$
be a locally standard $T$-chart.
After possibly shrinking $U_M$, we may assume
that $f|_{U_M}$ is a diffeomorphism of $U_M$
with an open subset $U_{M'}$ of $M'$.
The composition
$$ \phi' := \phi \circ f^{-1} \colon U_{M'} \to \Omega $$
is then a locally standard $T$-chart on $M'$.
Chasing the definitions, we obtain a diagram
$$ \xymatrix{
\calO \times \T^d \ar[r]^{\text{Id}} & \calO' \times \T^d \\
 P_{U_M} \ar[u]^{\xi} \ar[r]^{\tilde{f}} & P_{U_{M'}} \ar[u]_{\xi'}.
}
$$
where $\xi$ and $\xi'$ are diffeomorphisms.
Because such sets $P_{U_M}$ form an open cover of $P_M$, 
we conclude that $\tilde{f} \colon P \to P_{M'}$ 
is a diffeomorphism, as required.
\end{proof}

We are now ready to construct the simultaneous toric radial blowup functor
$\frakM \to \frakP$. 

\begin{construction} \labell{M to P}
We now construct a functor from the category $\frakM$ to the category $\frakP$.
To each locally standard $T$-manifold $M$, we associate
the principal $T$-bundle $\Pi \colon P_M \to Q$
as in \Cref{cor:structure}, 
where $Q := M/T$ is the 
decorated manifold-with-corners as in \Cref{def:functor}. 
To each equivariant diffeomorphism 
$f \colon M \to M'$
of locally standard $T$-manifolds, 
we associate the corresponding map $\tilde{f} \colon P_M \to P_{M'}$.
By \Cref{lift of diffeo}, $\tilde{f}$ is an equivariant local diffeomorphism.
\eoc
\end{construction}

\begin{lemma} \labell{natural} \ 
For each locally standard $T$-manifold~$M$,
there exists a (necessarily unique, necessarily equivariant) diffeomorphism
$\alpha_M \colon (P_M)_{\cut} \to M$
such that the following diagram commutes,
where $c$ is the cutting map.
\begin{equation} \labell{alpha M}
\xymatrix{
 P_M \ar[d]_{c_\odot} \ar[rr]^{c} && (P_M)_{\cut} \ar@{-->}[d]^{\alpha_M} \\ 
 M \ar[rr]^{\text{identity}} && M
}
\end{equation}
\end{lemma}

\begin{proof}
%\ynote{check:}
Because $c$ restricts to an equivariant homeomorphism 
between the preimages of $M_{\free}$ in $P_M$ and in $(P_M)_{\cut}$,
a continuous map $\alpha_M$ such that the diagram commutes
is necessarily unique and equivariant.

The blowup construction and the cutting construction
commute with open inclusions 
(\Cref{UM is UP cut}, \Cref{open in blowup}):
For any invariant open subset $U_M$ of $M$, 
\ $P_{U_M}$ is an open subset of $P_M$,
\ ${(P_{U_M})}_{\cut}$ is an open subset of ${(P_M)}_{\cut}$,
and the composition of 
the cutting map $c \colon P_{U_M} \to {(P_{U_M})}_{\cut}$
with the inclusion ${(P_{U_M})}_{\cut} \subset {(P_M)}_{\cut}$
is the restriction of the cutting map $c \colon P_M \to {(P_M)}_{\cut}$.
(Since the projection map is given by $\Pi = \pi \circ c_{\odot}$
(\Cref{free}),
for $U_Q = \pi(U_M)$,
the composition of the projection map $\Pi \colon P_{U_M} \to U_Q$
with the inclusion $U_Q \subset Q$
is the restriction of the projection map $\Pi \colon P_M \to Q$.)

Therefore, it is enough to show that, 
if $U_M \subset M$ is the domain of a locally standard $T$-chart,
then there exists a diffeomorphism $\alpha$
such that the following diagram commutes.
$$ \xymatrix{
 P_{U_M} \ar[d]_{c_\odot} \ar[rr]^{c} && (P_{U_M})_{\cut} 
\ar@{-->}[d]^{\alpha} \\ 
 U_M \ar[rr]^{\text{identity}} && U_M
}$$
Indeed, because such diffeomorphisms $\alpha$ coincide
on the intersections of their domains with the preimage of $M_{\free}$, 
by continuity they will fit together into a smooth map 
$\alpha_M \colon (P_M)_{\cut} \to M$
such that the diagram \eqref{alpha M} commutes.
%\ynote{ \\ Need to explain why this smooth map is a diffeomorphism.}

Let $ \phi \colon U_M \to \Omega $
be a locally standard $T$-chart, equivariant with respect to an isomorphism 
$ \rho \colon T \to \T^d $,
with $\Omega \subset \C^n \times \T^l \times \R^m$.
Let $U_Q := \pi(U_M)$, let $\calO := \theta(\Omega)$ 
be the image of $\Omega$ in $\Rplus^n \times \R^m$,
let $\psi \colon U_Q \to \calO$ the diffeomorphism that is induced from $\phi$
(\Cref{structure on chart}),
and let $\Pi = \pi \circ c_{\odot} \colon P_{U_M} \to U_Q$
be the projection map.
By \Cref{xi}, we obtain a commuting diagram
{
$$
\xymatrix{
 P_{U_M} \ar[d]_{c_{\odot}} \ar@/_2pc/[dd]_{\Pi} \ar[r]^{\xi}_{\cong}
 & **[r]{\calO \times \T^{n+l}} 
   \ar[d]^{c^\std} \ar@/^2.5pc/[dd]^{\text{projection}} \\
 U_M \ar[r]^{\phi} \ar[d]_{\pi} & \Omega \ar[d]^{\theta^\std} \\
 U_Q \ar[r]^{\psi} & \calO
}
$$
}
By the definition of the smooth structure on $P_M$,
the map $\xi$ is a $\rho$-equivariant diffeomorphism.
In particular, the map $\xi$ is a trivialization 
of the principal bundle $\Pi \colon P_{U_M} \to U_Q$.
Applying \Cref{cut smooth} to this trivialization,
and by the definition of the smooth structure on the cut space,
we obtain a diffeomorphism 
$\phi_{\text{new}} \colon {(P_{U_M})}_{\cut} \to \Omega$
such that the following diagram commutes,
where $c^\std$ is the model cutting map 
and $\theta^\std$ is the model quotient map.
(Here, the $U_P$ of \Cref{cut smooth} is our $P_{U_M}$,
the $\tilde{\psi}$ of \Cref{cut smooth} is our $\xi$, 
and the $U_M$ of \Cref{cut smooth} is our $(P_{U_M})_{\cut}$.
In particular the $U_M$ of \Cref{cut smooth} is not our $U_M$.)
{
$$ \xymatrix{
 P_{U_M} \ar[r]^{\xi} \ar[d]_{c} 
   & **[r] {\calO \times \T^{n+l}} \ar[d]^{c^\std} \\
 {(P_{U_M})}_{\cut} \ar[r]^{\phi_{\text{new}}} \ar[d] 
   & \Omega \ar[d]^{\theta^\std} \\
 U_Q \ar[r]^{\psi} & \calO
}
$$
}
Combining these last two diagrams, 
we obtain the following commuting diagram
where all the horizontal arrows are diffeomorphisms.
$$
\xymatrix{
 P_{U_M} \ar[r]^{\xi} \ar[d]_{c} 
   & {\quad \calO \times \T^{n+l}} \ar[d] 
 & P_{U_M} \ar[d]^{c_{\odot}} \ar[l]_{\quad \xi} \\
 {(P_{U_M})}_{\cut} \ar[r]^{\phi_{\text{new}}} \ar[d] 
   & \Omega \ar[d] 
 & U_M \ar[d]^{\pi} \ar[l]_{\phi} \\
 U_Q \ar[r]^{\psi} & \calO 
 & U_Q \ar[l]_{\psi} 
}
$$
The composition 
$$\alpha:= \phi^{-1} \circ \phi_{\text{new}} \colon (P_{U_M})_{\cut} \to U_M$$
is then an equivariant diffeomorphism, and we have a commuting diagram
$$ \xymatrix{
 & P_{U_M} \ar[dl]_{c} \ar[dr]^{c_{\odot}} & \\
 (P_{U_M})_{\cut} \ar[rr]^{\alpha} && U_M \, ,
}$$
as required.
\end{proof}

\begin{lemma} \labell{MPM to I}
The map $M \mapsto \alpha_M$ of 
Lemma~\ref{natural} is a natural isomorphism 
from the composition $\frakM_\invt \to \frakP_\invt \to \frakM_\invt$
to the identity functor on $\frakM_\invt$.
\end{lemma}

\begin{proof}
Let $f \colon M_1 \to M_2$ be an equivariant diffeomorphism,
let $\tilde{f} \colon P_{M_1} \to P_{M_2}$ be the diffeomorphism 
that is obtained from $f$ 
by the simultaneous toric radial blowup functor, and let 
$\tilde{f}_{\cut} \colon (P_{M_1})_{\cut} \to (P_{M_2})_{\cut}$
be the diffeomorphism that is induced from $\tilde{f}$ by the cutting functor.  
We need to show that 
$$\alpha_{M_2} \circ \tilde{f}_{\cut} = f \circ \alpha_{M_1}.$$
Consider the diagram 
$$
\xymatrix{
  P_{M_1} \ar[r]^{\tilde{f}} 
\ar@/_4pc/[dd]_{c_{\odot}}
\ar[d]_{c_1} 
& **[r] P_{M_2} 
\ar[d]^{c_2}  \ar@/^4pc/[dd]^{c_{\odot}}
\\
 (P_{M_1})_{\cut} \ar[r]^{{\tilde{f}}_{\cut}} 
\ar[d]_{\alpha_{M_1}}
& **[r] (P_{M_2})_{\cut} 
\ar[d]^{\alpha_{M_2}} 
\\
M_1 \ar[r]^{f} & **[r] M_2,
}
$$
where $c_1$ and $c_2$ are the cutting maps,
and where $\alpha_{M_1}$ and $\alpha_{M_2}$ are as in \Cref{natural}.
By \Cref{natural}, the left and right portions of the diagram commute,
so the big rectangle in the diagram commutes.
By properties of the cutting functor,
%\ynote{pointer?},
the top square of the diagram commutes.
Because the cutting maps $c_1$ and $c_2$ are onto, 
the bottom square of the diagram also commutes,
which is what we needed to show.
\end{proof}

\begin{corollary}
The cutting functor $\frakP_\invt \to \frakM_\invt$ is essentially surjective.
\end{corollary}

This completes the arguments that we needed
for our classification of locally standard $T$-manifolds.
We claim that, moreover, 
the cutting functor $\frakP \to \frakM$ is an isomorphism of categories.

\begin{lemma} \labell{natural2}
For each principal bundle $\Pi \colon P \to Q$
over a decorated manifold-with-corners, let $M := \Pcut$.  
Then there exists a (necessarily unique, necessarily equivariant) 
diffeomorphism
$\beta_P \colon P_M \to P$ 
such that the following diagram commutes, where $c$ is the cutting map.
$$ \xymatrix{
 P_M \ar[d]_{c_\odot} \ar[rr]^{\beta_P} && P\ar[d]^{c} \\
 M \ar[rr]^{\text{identity}} && M \,,
}$$
\end{lemma} 

\begin{proof}
%\ynote{check:}
This is a consequence of \Cref{cute diagram} and \Cref{cor:structure}. 
\end{proof}

\begin{lemma} \labell{PMP to I}
The map $P \mapsto \beta_P$ of \Cref{natural2}
is a natural isomorphism 
from the composed functor $\frakP \to \frakM \to \frakP$
to the identity functor on $\frakP$.
\end{lemma}

\begin{proof}
%\ynote{check:}
Let $\hat{\psi} \colon P_1 \to P_2$ be an equivariant diffeomorphism.
Let $\psi \colon M_1 \to M_2$ be the diffeomorphism
that is obtained from $\hat{\psi}$ by the cutting functor.
%\ynote{pointer}.
Let $\tpsi \colon P_{M_1} \to P_{M_2}$ be the diffeomorphism 
that is obtained from $\psi$
by the simultaneous toric radial blowup functor.
We need to show that
$$\hat{\psi} \circ \beta_{P_1} = \beta_{P_2} \circ \tilde{\psi}.$$
Consider the diagram
$$ \xymatrix{
 P_{M_1} \ar[rr]^{\tilde{\psi}} 
\ar[d]_{\beta_{P_1}} 
 \ar@/_4pc/[dd]_{c_{\odot}} 
  && P_{M_2} 
\ar[d]^{\beta_{P_2}}
 \ar@/^4pc/[dd]^{c_{\odot}} 
 \\
 P_1 \ar[d]_{c_1} \ar[rr]^{\hat{\psi}} && P_2 \ar[d]^{c_2} \\
 M_1 
\ar[rr]^{\psi}
   && M_2 \,,
}$$
where $c_1,c_2$ are the cutting maps,
and where $\beta_{P_1}$ and $\beta_{P_2}$ are as in \Cref{natural2}.
By \Cref{natural2}, the left and right portions of the diagram commute,
so the big rectangle in the diagram commutes.
By properties of the cutting functor,
%\ynote{pointer?},
the bottom square of the diagram commutes.
The cutting maps $c_1$ and $c_2$ are not smooth,
but they are diffeomorphisms on open dense subsets.
The maps $c_{\odot}$ are also diffeomorphisms on open dense subsets.
This implies that the top square of the diagram also commutes,
which is what we needed to show.
\end{proof}

%-------------------------------------------
\section{Relations with the literature}
%-------------------------------------------
\labell{sec:literature}

The relevant literature is vast.
We will sample here only a small number of earlier results.

First, we comment on the relation between our classification and Wiemeler's.

\begin{remark} \labell{rk:wiemeler-details}
If $M$ is a locally standard $T$-manifold that \emph{admits a section}, 
in the sense that there exists a continuous map $M/T \to M$
that associates to each orbit an element of that orbit,
then its Chern class is zero.
It follows that our functor restricts to a functor
from the subcategory of those locally standard $T$-manifolds 
that admit sections
to the subcategory of those triples $(Q,\lambda,c)$ where $c=0$,
with the same morphisms as before,
and that this restricted functor is full.
In particular,
for any two locally standard $T$-manifolds $M$ and $M'$ that admit sections, 
if there is a diffeomorphism $M/T \to M'/T$
that intertwines the unimodular labellings,
then there is an equivariant diffeomorphism $M \to M'$.
This recovers a result of Wiemeler \cite[Theorem~1.1]{Wi:classification}.
Conversely, if $c=0$, then $M$ admits a section (see \Cref{rk:section}).
Thus, the restricted functor is essentially surjective.
Namely, each triple $(Q,\lambda,c)$ with $c=0$ 
arises from some locally standard
$T$-manifold $M$ that admits a continuous section.
\eor
\end{remark}

\begin{remark} \labell{rk:section}
(Wiemeler, in \cite[p.~550, l.~7]{Wi:classification}, writes 
``Note that there is always a section to the orbit map
$M \to M/T$ if $H^2(M/T,\Z)=0$''.
Our work yields a proof:
If $c=0$, then the simultaneous toric radial blowup $P_M$ of $M$ 
admits a section, which descends to a section of its cut space
$(P_M)_{\cut}$, which by \Cref{natural} gives a section of $M$.
We note that it is \emph{not} true 
that every continuous section of $\Mfree \to \Mfree/T$ 
extends to a continuous section of $M \to M/T$.)
\eor
\end{remark}

\begin{remark} \labell{rk:wiemeler}
Wiemeler's work relies on the notion of a \emph{regular section} 
$s \colon M/T \to M$ for a locally standard $T$-manifold $M$.
This notion is defined in \cite[Def.~3.1]{Wi:classification}.
A-priori, the definition depends on choices of locally standard $T$-charts.
A-posteriori, it doesn't. 
This is stated in Lemma~3.2 of \cite{Wi:classification},
but the proof of this lemma 
only shows that the definitions coincide for two locally standard $T$-charts
with the same locally standard $T$ model
and with the same isomorphism 
$\rho_{\alpha_1,\ldots,\alpha_d} \colon T \to \T^d$.
In particular, it only considers the case $n=n'$ and $m=m'$.
Our \Cref{descend diffeo from P to M} completes Wiemeler's argument
and confirms that Wiemeler's notion of a ``regular section''
is indeed independent of the choice of a locally standard $T$-chart.
\eor
\end{remark}

We also recover earlier results of Davis:

\begin{remark} \labell{rk:Davis}
Let $Y$ be a $\T^m$ manifold, modeled on the standard representation
in the sense that at each point the stabilizer 
and its action on the normal to the orbit-type stratum
occur among those of the standard $\T^m$ action on $\C^m$.
Let $Y_{\free}$ be the set of points where the $\T^m$ action is free;
assume that the principal bundle $Y_{\free}$ is trivial.
The quotient $P = Y/\T^m$
is an $n$-dimensional manifold-with-faces (\Cref{rk:mfld-w-faces}).
For $i=1,\ldots,m$, let $F_i$ be the union of those facets of $P$
whose preimage in $Y$ is fixed by the $i$th standard sub-circle of $\T^m$.
Let $f \colon P \to \Rplus^m$ be a smooth map such that $\{f_i=0\} = F_i$ 
and $df_i \neq 0$ along $F_i$ for all $i$
The {\it moment-angle manifold} $\mathcal{Z}_{P}$ 
is defined by the pull-back diagram:
\begin{align*}
\xymatrix{
\mathcal{Z}_{P}\ar[r]\ar[d] & \mathbb{C}^{m}\ar[d]^{(z_1,\ldots,z_m)
 \mapsto (|z_1|^2,\ldots,|z_m|^2)} \\
P\ar[r]^{f} & \mathbb{R}_{+}^{m}.
}
\end{align*}
Also see \cite{BP}.
According to Davis \cite[Proposition 6.2]{Da13}
(which relies on his former works \cite{Da78, Da78_note, Da81}),
$Y$ is equivariantly diffeomorphic to the moment-angle manifold 
$\mathcal{Z}_P$. 
This result also follows from our theorem.
Here, the Chern class is zero, and the unimodular labelling 
takes each point in the $i$th stratum
to $\pm e_i$
where $e_1,\ldots,e_m$ is the standard basis of 
the weight lattice $\mathfrak{t}_{\mathbb{Z}}\simeq \mathbb{Z}^{m}$.

As an application, Davis \cite[Proposition 6.4 (iii)]{Da13}
obtains unique-up-to-diffeomorphism smooth structures on quasitoric manifolds.
To construct a quasitoric manifold, following \cite{DJ},
we start with an $n$ dimensional torus $T$
and an $n$ dimensional simple convex polytope $P$
with $m$ facets $F_1,\ldots,F_m$.
To each facet $F_i$, we associate a (parametrized) sub-circle of $T$,
such that the resulting map 
$\lambdahat \colon \mathring{\del}^1 P \to \tZhat$
is a unimodular labelling in our sense.
Consider the homomorphism $\lambda \colon \T^m \to T$
whose restriction to the $i$th component is the corresponding sub-circle
of $T$. Then $H := \ker \lambda$ acts freely on ${\mathcal Z}_P$.
The corresponding quasi-toric manifold is ${\mathcal Z}_P/H$,
with its induced $T \cong \T^m/H$ action.
According to Davis \cite[Proposition 6.4 (iii)]{Da13},
if $M$ is a $2n$ dimensional locally standard $T$ manifold
and there exists a diffeomorphism 
$M/T \cong P$ that respects the unimodular labelling,
then $M$ is equivariantly diffeomorphic to the quasitoric manifold 
${\mathcal Z}_P/H$.
This result also follows from our theorem.
\eor
\end{remark} 

Inspired by Bredon, we obtain a relation between different smooth structures
on topological locally standard $T$-manifolds:

\begin{remark} \labell{rk:bredon}
Bredon works with compact groups $G$ that are not necessarily tori,
and restricts to so-called \emph{special} $G$-manifolds,
which have only two orbit types
and whose quotients $M/G$ are manifold-with-boundary.
According to his theorem \cite[Chapter VI, Theorem 6.3]{Br},
a smooth structure on $M/G$ determines a smooth structure on $M$,
unique up to equivariant diffeomorphism.

A \emph{topological locally standard $T$-manifold}
is a topological manifold $M$, equipped with an action of the torus $T$,
that can be covered by the domains of topological locally standard
$T$-charts (\Cref{def:Tchart v2}).
For a locally standard topological $T$-manifold $M$,
the quotient $Q := M/T$ is a topological manifold-with-boundary
(topologically we don't see the corners),
equipped with a Chern class $c \in H^2(Q;\tZ)$
and a unimodular labelling $\lambdahat \colon \mathring{\del}^1 Q \to \tZhat$
that are defined as in~\Cref{sec:defs},
where $\mathring{\del}^1 Q$ is the set of points in $Q$
whose preimages in $M$ have a one dimensional stabilizer.

Our classification implies the following uniqueness result.
Let $M_1$ and $M_2$ be two (smooth) locally standard $T$-manifolds,
with the same underlying  topological locally standard $T$-manifold $M$,
and that induce the same manifold-with-corners structure on $M/T$.
Then the identity map on $M/T$ lifts to an equivariant diffeomorphism
from $M_1$ to $M_2$.
\eor
\end{remark}

We conclude with a few more pointers to the literature.

\begin{remark} \ 
\begin{enumerate}
\item
Other authors --- for example, J\"anich \cite{janich},
Albin-Melrose \cite{AM}, and Davis \cite{Da78}
 --- have used \emph{iterated} radial blowups to describe compact group actions.
In contrast to iterated radial blowups,
our \emph{simultaneous} blowup does not affect the quotient space.
We are not aware of usages of a simultaneous radial blowup
of the form that we have in this paper.

\item
For a radial blowup along one toric divisor,
our construction agrees with J\"anich's \cite{janich}
as topological spaces but not as (smooth) manifolds.
See \cite{cutting}, where we used the terms
``radial blowup" and ``radial-squared blowups''
to distinguish the two smooth structures on the resulting spaces.

\item
In \Cref{sec:example} we obtained lens spaces 
as three dimensional locally standard $T$-manifolds.
Constructions of higher dimensional lens spaces
in the context of toric topology appeared in \cite{lens}. 

\item
Let $M$ be a compact connected orientable manifold $M$ 
with a faithful torus action.
Suppose that $\dim M = 2 \dim T$ and that the odd cohomology of $M$ 
(over $\Z$) vanishes.
Then $M$ is a locally standard $T$-manifold.
Moreover, for every subgroup $H$ of $T$,
every component $N$ of the set $N^H$ of $H$-fixed points
contains a $T$-fixed point.  See \cite[Theorem 4.1]{MP}.

\item
Related classification results in the symplectic context
appeared in these papers:
For torus actions:
\cite{De}, %Delzant 1988
\cite{KL}. %Karshon-Lerman
For completely integrable systems:
\cite{duist}. %Duistermaat 1980
\cite{Mol}, % Mol (2024)
\cite{zung}. % Zung, 2003
Completely integrable systems can be viewed as singular Lagrangian 
torus fibrations $M \to B$.
Completely integrable systems with elliptic singularities
are examples of (smooth) locally standard local torus actions
(cf.~\cite{Yo}),
where the torus that acts on $M$ varies along $B$.  

\item
A classification result in a complex analytic context 
appeared in \cite{IK}.

\item
The following toric topology papers connect with our work in various ways
and provided inspiration:
\cite{HM}, % Hattori-Masuda, theory of multi-fans, 2003
\cite{Ku}, % Kuroki, Orlik-Raymond type, 2016
\cite{MP}, % Masuda-Panov, 2006
\cite{sarkar-song}, % Sarkar-Song, 
%hypercharacteristic function ~ unimodular labelling, 2021
\cite{Yo}, % Yoshida, 2011
as well as the earlier paper 
\cite{OR}. % Orlik and Raymond

\item
Characteristic classes, such as Chern classes,
appeared as measurements of ``twistedness'' of $T$-spaces in 
\cite{duist}, \cite{HS}, \cite{zung}, 
\cite{KL}, \cite{Yo}, \cite{Mol}, \cite{Wi:classification}.

\end{enumerate}
\end{remark}

}

\end{document}